\documentclass[reqno]{amsart}
\usepackage[T1]{fontenc}
\usepackage[utf8]{inputenc}
\usepackage[dvipsnames]{xcolor}
\usepackage[hidelinks]{hyperref}
\usepackage[maxbibnames=99, style=alphabetic]{biblatex}
\usepackage{tikz}\usetikzlibrary{cd}
\usepackage{amsmath,amsthm,amssymb,bbm}
\usepackage{mathtools}
\usepackage{lmodern}
\usepackage{etex}
\usepackage{enumitem}
\usepackage{adjustbox}
\usepackage{ifthen}
\usepackage{mathrsfs}
\usepackage{libertine}
\usepackage{libertinust1math}
\usepackage{amsthm}
\usepackage{csquotes}
\definecolor{NARANJA}{rgb}{1,0.467,0}
\definecolor{AZUL}{rgb}{0,0.53,1}
\hypersetup{
    colorlinks=true,
    citecolor=AZUL,
    linkcolor=NARANJA,
    urlcolor=AZUL,
}

\newtheoremstyle{dotlessB}{}{}{}{}{\bfseries}{}{ }{}
\theoremstyle{dotlessB}
\newtheorem{Dfn}{Definition}[section]
\newtheorem{Exp}[Dfn]{Example}
\newtheorem{Rmk}[Dfn]{Remark}
\newtheorem{Lmm}[Dfn]{Lemma}
\newtheorem{Prp}[Dfn]{Proposition}
\newtheorem{Prp-Dfn}[Dfn]{Proposition-Definition}
\newtheorem{Thm}[Dfn]{Theorem}
\newtheorem{Cor}[Dfn]{Corollary}

\newtheorem{Cdt}[Dfn]{Condition}

\newtheorem*{Cvt}{Conventions}
\newcommand{\DfnLbl}{(\alph*)}
\newcommand{\ExpLbl}{(\arabic*)}
\newcommand{\RmkLbl}{(\arabic*)}
\newcommand{\LmmLbl}{(\alph*)}
\newcommand{\PrpLbl}{(\alph*)}
\newcommand{\ThmLbl}{(\alph*)}

\newtheorem{lThm}{Theorem}           
 
\newcommand{\FDA}{\Lambda}           
\newcommand{\ACF}{\mathbbm{k}}       
\newcommand{\AUR}{R}                 
\newcommand{\KbI}[1]{\homotopyCat^b(\text{Inj}#1)}
\newcommand{\Kzn}[1]{\homotopyCat^{[0,n]}(\text{Inj}#1)}

\newcommand{\subcat}{\mathcal{S}}
\newcommand{\homotopyCat}{\mathcal{K}}
\newcommand{\derivedCat}{\text{D}}
\newcommand{\extriangulatedCat}{\mathscr{C}}
\newcommand{\extFunc}{\mathbb{E}}
\newcommand{\additiveReal}{\mathfrak{s}}
\newcommand{\extriangulatedTriple}{(\extriangulatedCat,\extFunc,\additiveReal)}
\newcommand{\torsionClass}{\mathcal{T}}
\newcommand{\torsionfreeClass}{\mathcal{F}}
\newcommand{\torsionPair}{(\torsionClass,\torsionfreeClass)}
\newcommand{\cotorsionClass}{\mathcal{X}}
\newcommand{\cotorsionfreeClass}{\mathcal{Y}}
\newcommand{\cotorsionPair}{(\cotorsionClass,\cotorsionfreeClass)}
\newcommand{\cotorsionKernel}{\cotorsionClass\cap\cotorsionfreeClass}
\newcommand\syz[2][1]{\ifthenelse{\equal{#1}{1}}{\Omega^1(#2)}{\Omega^{#1}(#2)}}
\newcommand\cosyz[2][1]{\ifthenelse{\equal{#1}{1}}{\Sigma^1(#2)}{\Sigma^{#1}(#2)}}
\newcommand{\triangulatedCat}{\mathscr{T}}
\newcommand\shiftFunc[2][1]{\ifthenelse{\equal{#1}{1}}{#2[1]}{#2[#1]}} 
\newcommand{\triangleClass}{\Delta} 
\newcommand{\triangulatedTriple}{(\triangulatedCat,\shiftFunc{},\triangleClass)}

\title{Large silting mutation in extriangulated categories}
\author{Diego Alberto Barcel\'o Nieves}
\address{Universit\`a degli Studi di Verona, Strada le Grazie 15, 37134 Verona, Italia}
\email{\href{mailto:diegoalberto.barcelonieves@univr.it}{diegoalberto.barcelonieves@univr.it}}
\date{}

\begin{document}

\begin{abstract}
    Silting mutation in triangulated categories\textemdash both at the level of objects and of subcategories\textemdash was introduced in \cite[]{aihara2012silting}, and later generalized to extriangulated categories in \cite[]{adachi2025assortment}. It simultaneously encompasses the mutation theories of cluster-tilting objects in cluster theory and of compact 2-term silting complexes and support $\tau$-tilting modules in $\tau$-tilting theory. In this article, we develop an infinite-dimensional analog of silting mutation in extriangulated categories with set-indexed (co)products, which we then apply to obtain a theory of mutation for $n$-cosilting complexes over an arbitrary ring, as well as for infinite-dimensional $n$-(co)tilting modules over a ring of finite global dimension. The former theory is also shown to reinterpret the cosilting mutation introduced in \cite[]{hügel2025mutation}.
    \noindent \textbf{Keywords.} Infinite-dimensional mutation, silting subcategories, extriangulated categories, hereditary complete cotorsion pairs, bounded cosilting complexes
\end{abstract}

\maketitle
\tableofcontents

\section{Introduction}\label{sec: introduction}

In Homological Algebra, a torsion pair allows one to decompose a category into a pair of semi-orthogonal subcategories\textemdash in the sense of there being no non-zero morphisms from one of them to the other\textemdash which preserve much of its homological information via conflations, since the whole category can be reconstructed by taking extensions between them. Each torsion pair provides a unique decomposition, so being able to control them yields a powerful tool for the study of the category as a whole. In \cite[]{adachi2014tau} (see also \cite[]{derksen2015general}) the authors proved that, for a finite dimensional algebra $\FDA$ over an algebraically closed field, the class of \emph{functorially finite} torsion pairs in $\text{mod}(\FDA)$ is in one-to-one correspondence with (compact) 2-term silting complexes over $\FDA$, which have an associated operation of mutation\textemdash i.e., a procedure by which a (direct) summand of a 2-term silting complex is replaced with an object in such a way that the resulting complex is also 2-term silting. The present work is part of an ongoing collective effort to extend the mutation of functorially finite torsion pairs to a more extensive class of torsion pairs. The motivation is that, for any ring $\AUR$, a similar one-to-one correspondence between arbitrary torsion pairs in $\text{mod}(\AUR)$ and 2-term cosilting complexes over $\AUR$ follows from \cite[Lemma~4.4]{crawley1994locally}, \cite[Corollary~3.5~\&~Corollary~4.8]{breaz2017cosilting} and \cite[Corollary~3.9]{hügel2018abundance} (see also \cite[Theorem~4.16~\&~Theorem~4.18]{zhang2017cosilting}), which begs the question: \emph{How can we mutate the latter?}

In recent work, two approaches to this question have been introduced, which we briefly describe (see also the survey \cite[]{hügel2026torsion}). In \cite[]{hügel2024torsion}\cite[]{hügel2026mutation}, a one-to-one correspondence between 2-term cosilting complexes and maximally rigid sets of the Ziegler spectrum is established, an operation of mutation for such sets is defined, and mutability criteria are given in terms of topological properties of this spectrum. In \cite[]{hügel2025mutation}, an operation of \emph{cosilting mutation} is defined for (the much more general) cosilting objects in triangulated categories, each of which gives rise to a $t$-structure and a hereditary torsion pair of its heart. In the case of pure-injective cosilting objects in compactly generated triangulated categories (which includes 2-term cosilting complexes), mutability criteria are given in terms of closure properties of the associated hereditary torsion pair, as well as the existence of certain approximations in the triangulated category; crucially, however, this does not ensure that the mutation of a 2-term cosilting complex is once again 2-term\textemdash and thus, corresponds to a torsion pair. We also note that, even though cosilting mutation is a vast generalization of compact silting mutation, the latter may be understood at the level of silting subcategories in triangulated categories (see \cite[]{aihara2012silting}), whereas no similar description is currently known for the former.

In this work, we introduce a different framework for the mutation of 2-term cosilting complexes by generalizing the mutation of silting subcategories\textemdash which typically have only been considered to be closed under summands and finite direct sums\textemdash to silting subcategories which are moreover closed under set-indexed products. This is inspired by a characterization of 2-term cosilting complexes as those whose closure under summands and set-indexed products is a silting subcategory of the homotopy category of 2-term complexes of injective modules (Proposition~\ref{prp: n-cosilting} for $n=1$). Remarkably, this category is not triangulated due to not being closed under shifts; however, it is extriangulated by virtue of being closed under extensions in the derived category. This forces us to develop our theory of mutation for product-closed silting subcategories\textemdash which we call \emph{large silting mutation}\textemdash in the general context of extriangulated categories, following \cite[\S~4.2]{adachi2025assortment}. 

Let us briefly describe this operation. Given product-closed additive subcategories $\mathcal{D}$ and $\subcat$ of some extriangulated category such that $\subcat$ is presilting, $\mathcal{D}\subseteq\subcat$, and each object in $\subcat\setminus\mathcal{D}$ has a right $\mathcal{D}$-approximation for which a cocone exists\textemdash in which case we call $\mathcal{D}$ a \emph{good contravariantly finite subcategory} of $\subcat$ (Definition~\ref{dfn: approximations})\textemdash, we define $\prod\nolimits^R(\subcat;\mathcal{D})$ by substituting the aforementioned objects in $\subcat$ with the cocones of their right $\mathcal{D}$-approximations and closing under products and summands (Definition~\ref{dfn: product-mutation}). Dually, we define a \emph{good covariantly finite subcategory} $\mathcal{D}$ of $\subcat$, and subsequently $\prod\nolimits^L(\subcat;\mathcal{D})$. Our first main result establishes the general conditions under which this operation satisfies the properties expected from silting mutation, and provides an extra condition under which it can be linked to the mutation of certain objects in the category, following \cite[\S 3]{hügel2025mutation}.

\begin{lThm}[Theorem~\ref{thm: mutation}~\&~Proposition~\ref{prp: producer mutation}]\label{thm: A}
    Let $\extriangulatedTriple$ be an extriangulated category with positive extensions and exact products. 
    \begin{enumerate}[label=\ThmLbl]
        \item Let $\subcat$ be a presilting subcategory of $\extriangulatedCat$ and $\mathcal{D}$ be a good contravariantly finite subcategory of $\subcat$ such that $\text{Prod}(\subcat)=\subcat$ and $\text{Prod}(\mathcal{D})=\mathcal{D}$. Then the following statements hold.
            \begin{itemize}[label=\ThmLbl]
                \item[(a1)] $\prod\nolimits^R(\subcat;\mathcal{D})$ is a product-closed presilting subcategory of $\extriangulatedCat$.
                \item[(a2)] $\prod\nolimits^R(\subcat;\mathcal{D})\geqslant\subcat$, where the equality holds if, and only if, $\subcat=\mathcal{D}$.
                \item[(a3)] If $\subcat$ is a silting subcategory of $\extriangulatedCat$, then so is $\prod\nolimits^R(\subcat;\mathcal{D})$.
                \item[(a4)] $\mathcal{D}$ is a good covariantly finite subcategory of $\prod\nolimits^R(\subcat;\mathcal{D})$ and
                    \[
                        \prod\nolimits^L\big(\prod\nolimits^R(\subcat;\mathcal{D});\mathcal{D}\big)=\subcat.
                    \] 
            \end{itemize}
            Moreover, if $\mathcal{D}$ is instead a good covariantly finite subcategory of $\subcat$, then the corresponding statements for $\prod\nolimits^L(\subcat;\mathcal{D})$ hold.

        \item Suppose that, for each $X\in\extriangulatedCat$, we have that $\text{Prod}(X^{\perp_{>0}})=X^{\perp_{>0}}$.
            Let $C,C'\in\extriangulatedCat$ be such that $\text{Prod}(C)$ and $\text{Prod}(C')$ are silting subcategories of $\extriangulatedCat$ and $\mathcal{D}=\text{Prod}(C)\cap\text{Prod}(C')$. Then the following conditions are equivalent.
            \begin{itemize}
                \item[(b1)] There exists an $\additiveReal$-conflation
                    \[
                        C\xrightarrow[]{f} B_0\to B_1
                    \] 
                    such that $f$ is a left $\mathcal{D}$-approximation of $C$ and $\text{Prod}(B_0\bigoplus B_1)=\text{Prod}(C')$.
                \item[(b2)] There exists an $\additiveReal$-conflation
                    \[
                        A_1\to A_0\xrightarrow[]{g} C'
                    \] 
                    such that $g$ is a right $\mathcal{D}$-approximation of $C'$ and $\text{Prod}(A_1\bigoplus A_0)=\text{Prod}(C)$.
                \item[(b3)] $\mathcal{D}$ is a good contravariantly finite subcategory of $\text{Prod}(C')$ and
                    \[
                        \prod\nolimits^R(\text{Prod}(C');\mathcal{D})=\text{Prod}(C).
                    \] 
                \item[(b4)] $\mathcal{D}$ is a good covariantly finite subcategory of $\text{Prod}(C)$ and
                    \[
                        \prod\nolimits^L(\text{Prod}(C);\mathcal{D})=\text{Prod}(C').
                    \] 
            \end{itemize}
    \end{enumerate}
\end{lThm}

In order to exploit Theorem~\ref{thm: A} to the fullest, we give additional conditions under which each product-closed silting subcategory $\subcat$ is in one-to-one correspondence with an equivalence class of objects $C$ such that $\subcat=\text{Prod}(C)$, which we call its \emph{producers}.

\begin{lThm}[Proposition~\ref{prp: bijections},~Lemma~\ref{lmm: equivalence}~\&~Lemma~\ref{lmm: bounded}]\label{thm: B}
    Let $\extriangulatedTriple$ be an extriangulated category with exact products, an $\extFunc$-injective cogenerator $Q$, and all objects of finite $\extFunc$-injective dimension. Then there exist mutually inverse bijections
    \begin{center}
        \begin{tikzcd}
            \big\{ \text{objects } C \text{ in } \extriangulatedCat \mid \text{Prod}(C) \text{ is a silting subcategory of } \extriangulatedCat \big\}_{/\sim} \arrow[shift right]{d}[]{}  \arrow[phantom, very near start, xshift=-10mm, yshift=1.25mm]{d}[]{\scriptstyle [C]} \arrow[mapsto, shorten=1.5mm, xshift=-10mm, yshift=0mm]{d}[]{} \\
            \big\{ \text{silting subcategories } \subcat \text{ of } \extriangulatedCat \mid \text{Prod}(\subcat)=\subcat \big\} \arrow[shift right]{d}[swap]{\Psi} \arrow[shift right]{u}[swap]{\Theta} \arrow[phantom, very near start, xshift=-10mm, yshift=-1mm]{u}[]{\scriptstyle \text{Prod}(C)} \\
            \big\{ \text{ hereditary complete cotorsion pairs } \cotorsionPair \text{ in } \extriangulatedCat \mid \text{Prod}(\cotorsionClass)=\cotorsionClass \text{ and } \cotorsionClass^\land=\extriangulatedCat \big\} \arrow[shift right]{u}[swap]{\Phi},
        \end{tikzcd}
    \end{center}
    where $\Psi(\subcat):=(\subcat^\lor,\subcat^\land)$, $\Phi\cotorsionPair:=\cotorsionKernel$, and a representative of the equivalence class $\Theta(\subcat)$ is given by considering any finite $\subcat$-resolution of $Q$ and taking the direct sum over the extensions which appear in the resolution.
\end{lThm}

We then give a canonical construction of an extriangulated category satisfying all of the hypotheses from Theorem~\ref{thm: A} and Theorem~\ref{thm: B} (Corollary~\ref{cor: construction}), which can be readily obtained from any cosilting object in a triangulated category with products. Important examples of this construction include, for any $n\in\mathbb{Z}_{>0}$, the homotopy category of $(n+1)$-term complexes of injective modules over a ring $\AUR$\textemdash whose producers coincide with $(n+1)$-term cosilting complexes over $\AUR$ (Proposition~\ref{prp: n-cosilting})\textemdash as well as their analogs over a connective dg algebra (Remark~\ref{rmk: dga}), thus establishing a connection with the deformation theory of triangulated categories introduced in \cite[]{genovese2021t-structures}. In particular, the $n=1$ case provides mutability criteria for 2-term cosilting complexes over $R$ in terms of the existence of certain approximations in the aforementioned homotopy category, both at the level of product-closed silting subcategories and of their producers. These results also reveal that cosilting mutation between bounded cosilting complexes may be reinterpreted as an infinite-dimensional analog of silting mutation between associated product-closed silting subcategories in an appropriate extriangulated category. 

As a final example, Theorem~\ref{thm: B} is shown to also be applicable to $\text{Mod}(\AUR)$ for $\AUR$ a ring of finite global dimension $n$, in which case the producers of the product-closed silting subcategories of $\text{Mod}(\AUR)$ coincide with $n$-cotilting modules over $\AUR$; hence, its $n$-cotilting modules can be mutated directly in the module category via large silting mutation. We remark that analogs of Theorem~\ref{thm: A} and \ref{thm: B} for coproduct-closed silting subcategories $\subcat$ and their \emph{coproducers} (i.e., objects $C$ such that $\subcat=\text{Add}(C)$) are valid for extriangulated categories satisfying analogous conditions for set-indexed coproducts instead of products.

The document is organized as follows. After covering the necessary preliminaries in Section~\ref{sec: preliminaries}, we develop the theory of mutation for product-closed silting subcategories in extriangulated categories in Section~\ref{ssec: product-closed}. We then establish Theorem~\ref{thm: B} in Section~\ref{ssec: producers} and show how, under an additional hypothesis, the aforementioned theory can be used to mutate producers of product-closed silting subcategories. Section~\ref{ssec: special} focuses on some special cases in which Theorem~\ref{thm: B} can be given a more detailed description, whereas Section~\ref{ssec: construction} presents a canonical construction of an example satisfying all of the hypotheses considered throughout Section~\ref{sec: large}. Finally, the examples of $n$-cosilting complexes over a ring and (infinitely-generated) $n$-(co)tilting modules over a ring of finite global dimension $n$ are treated in Section~\ref{ssec: n-cosilting} and Section~\ref{ssec: n-(co)tilting}, respectively. 

\section{Preliminaries on extriangulated categories}\label{sec: preliminaries}

\hypertarget{Cvt}{
    \begin{Cvt}
        Throughout, we assume that all subcategories are full and closed under isomorphisms in their ambient category. We thus make no distinction between a given subcategory and its corresponding class of objects. Moreover, we assume that any subcategory of an additive category is itself additive and that all of its finite direct sums are also direct sums in their ambient category, i.e., that it is an additive subcategory. Furthermore, we assume that all (co)products and families of objects are indexed by sets. We often denote singletons by their unique element and use the term ``large'' to refer to subcategories which contain infinite products or coproducts, depending on the context, or to make reference to modules which are not finitely generated. 
    \end{Cvt}
}

Recall that extriangulated categories are a class of additive categories with additional structure which simultaneously generalize both exact and triangulated categories, and that the extriangulated generalization of both kernel-cokernel pairs (hence, short exact sequences) and distinguished triangles are called \emph{conflations}. For details, we refer to \cite[]{nakaoka2019extriangulated}.

\subsection{Terminology and cotorsion pairs}\label{ssec: terminology}

From now on, let $\extriangulatedTriple$ be an extriangulated category. This setting is suitable for considering complete cotorsion pairs.

\begin{Dfn}\cite[Definition~2.15~\&~Definition~3.9]{nakaoka2019extriangulated}\label{dfn: terminology}
    Let $A\xrightarrow[]{f} B\xrightarrow[]{g} C$ be an $\additiveReal$-conflation.
    \begin{enumerate}[label=\DfnLbl]
        \item $f$ is an $\additiveReal$-\emph{inflation} and $C$ is a \emph{cone} (of $f$).
        \item $g$ is an $\additiveReal$-\emph{deflation} and $A$ is a \emph{cocone} (of $g$).
        \item $B$ is an \emph{extension} (of $A$ by $C$).
    \end{enumerate}
\end{Dfn}

\begin{Prp}\cite[Proposition~3.3~\&~Proposition~3.11]{nakaoka2019extriangulated}\label{prp: sequences}
    For each $\additiveReal$-conflation $A\to B\to C$ and $X\in\extriangulatedCat$, there exist exact sequences in Ab of the form:
    \begin{align*}
        \text{Hom}_\extriangulatedCat(X,A)\to \text{Hom}_\extriangulatedCat(X,B)\to \text{Hom}_\extriangulatedCat(X,C)\to \extFunc(X,A)\to \extFunc(X,B) \to \extFunc(X,C), \\
        \text{Hom}_\extriangulatedCat(C,X)\to \text{Hom}_\extriangulatedCat(B,X)\to \text{Hom}_\extriangulatedCat(A,X)\to \extFunc(C,X)\to \extFunc(B,X) \to \extFunc(A,X).
    \end{align*}
\end{Prp}

\begin{Dfn}\label{dfn: subcategories}
    Let $\subcat\subseteq \extriangulatedCat$. 
    \begin{enumerate}[label=\DfnLbl]
        \item We define the class of objects left orthogonal to $\subcat$ with respect to $\extFunc$ as
            \[
            {}^{\perp_1}\subcat:= \{ C\in\extriangulatedCat \mid \extFunc(C,S)=0 \text{ for all } S\in\subcat\}.
            \] 
            The class $\subcat^{\perp_1}$ of objects right orthogonal to $\subcat$ with respect to $\extFunc$ is defined dually.

        \item An $\subcat$-\emph{resolution} of an object $C\in\extriangulatedCat$ is a sequence $\{L_{i+1}\to S_i\to L_i\}_{i\in\mathbb{Z}_{\ge0}}$ of $\additiveReal$-conflations such that $L_0=C$ and $S_i\in\subcat$ for each $i\in\mathbb{Z}_{\ge0}$. If, furthermore, there exists an integer $k\ge0$ such that $L_k\not\simeq0$ and $L_{k+1}=0$, and therefore a smallest such integer $m$, then the \emph{length} of the $\subcat$-resolution is $m$; otherwise, it is $\infty$. An $\subcat$-\emph{coresolution of length} $m\in\mathbb{Z}_{\ge 0}\cup\{\infty\}$ of an object $C\in\extriangulatedCat$ is defined dually.

        \item We recursively define the subcategories of objects in $\extriangulatedCat$ with $\subcat$-resolutions of length at most $m\in\mathbb{Z}_{\ge0}$ as
    \[
        \subcat^\land_m := \begin{cases} \subcat &\text{if } m=0, \\ \{C\in\extriangulatedCat \mid \exists \text{ an }\additiveReal\text{-conflation } S'\to S\to C \text{ with } S\in\subcat \text{ and } S'\in\subcat^\land_{m-1}\} &\text{if } m>0, \end{cases}
    \] 
    and define the subcategory of objects in $\extriangulatedCat$ with finite $\subcat$-resolutions as
    \[
        \subcat^\land := \bigcup_{m\ge 0} \subcat^\land_m.
    \] 
    The subcategories $\subcat^\lor_m$, of objects in $\extriangulatedCat$ with $\subcat$-coresolutions of length at most $m\in\mathbb{Z}_{\ge0}$,  and $\subcat^\lor$, of objects in $\extriangulatedCat$ with finite $\subcat$-coresolutions, are defined dually. 
    
        \item We say that $\subcat$ is \emph{closed under}
            \begin{itemize}
                \item \emph{(co)cones} in $\extriangulatedCat$ if, for each $\additiveReal$-conflation $A\to B\to C$ with $A,B\in\subcat$ (respectively, with $C,B\in\subcat$), it follows that $C\in\subcat$ (respectively, that $A\in\subcat$);
                \item \emph{extensions} in $\extriangulatedCat$ if, for each $\additiveReal$-conflation $A\to B\to C$ with $A,C\in\subcat$, it follows that $B\in\subcat$;

                \item \emph{summands} in $\extriangulatedCat$ if, for each $\additiveReal$-conflation $A\to B\to C$ such that $\extFunc(C,A)=0$ (equivalently, for each $B\simeq A\bigoplus C$ in $\extriangulatedCat$) with $B\in\subcat$, it follows that $A,C\in\subcat$.
            \end{itemize}

        \item We say that the intersection of all subcategories of $\extriangulatedCat$ which contain $\subcat$ and satisfy some set of closure conditions is the \emph{closure} of $\subcat$ under the respective condition(s).

        \item We denote the closure of $\subcat$ under cones, cocones, extensions and summands in $\extriangulatedCat$ by $\text{thick}(\subcat)$, and we say that $\subcat$ is a \emph{thick subcategory} of $\extriangulatedCat$ if $\text{thick}(\subcat)=\subcat$.

        \item For $\subcat'\subseteq\extriangulatedCat$, we define the subcategory of extensions in $\extriangulatedCat$ of objects in $\subcat$ by objects in $\subcat'$ as
            \[
                \subcat\ast\subcat' = \{C\in\extriangulatedCat \mid \exists \ \text{ an } \additiveReal\text{-conflation } S\to C\to S' \text{ with } S\in\subcat \text{ and } S'\in\subcat' \}.
            \] 
    \end{enumerate}
\end{Dfn}

\begin{Rmk}\label{rmk: subcategories}
    Let $\subcat\subseteq \extriangulatedCat$.
    \begin{enumerate}[label=\RmkLbl]
        \item It follows from Proposition~\ref{prp: sequences} that ${}^{\perp_1}\subcat$ is closed under extensions in $\extriangulatedCat$. Moreover, by the additivity of the bifunctor $\extFunc$, we have that ${}^{\perp_1}\subcat$ is also closed under summands in $\extriangulatedCat$. Analogously, $\subcat^{\perp_1}$ has the same closure properties.

        \item If $\extriangulatedCat$ is triangulated then, by rotating the trivial distinguished triangles given by the identity morphisms, the conditions of closure under cones and cocones in $\extriangulatedCat$ imply those of closure under positive and negative shifts, respectively. In this case, the previous definition of thick subcategory coincides with the usual one in triangulated categories, as noted in \cite[Example~4.8]{nakaoka2022localization}, and $\text{thick}(\subcat)$ is precisely the smallest triangulated subcategory of $\extriangulatedCat$ which contains $\subcat$ and is closed under summands in $\extriangulatedCat$.

        \item If we consider $\ast$ as a binary operation on subcategories of $\extriangulatedCat$, then it is associative by the analog of the octahedral axiom for extriangulated categories and its dual.
    \end{enumerate}
\end{Rmk}

\begin{Prp-Dfn}\cite[Proposition~3.24~\&~Definition~3.25]{nakaoka2019extriangulated}\label{dfn: E-projectives}
    \begin{enumerate}[label=\DfnLbl]
        \item We define the subcategory of $\extFunc$-\emph{projective objects} in $\extriangulatedCat$ as $\text{Proj}_\extFunc(\extriangulatedCat):={}^{\perp_1}\extriangulatedCat$ and say that $\extriangulatedCat$ \emph{has enough $\extFunc$-projectives} if, for each $C\in\extriangulatedCat$, there exists an $\additiveReal$-deflation $P\to C$ with $P\in\text{Proj}_\extFunc(\extriangulatedCat)$. 

        \item We define the subcategory of $\extFunc$-\emph{injective objects} in $\extriangulatedCat$ as $\extriangulatedCat^{\perp_1}=:\text{Inj}_\extFunc(\extriangulatedCat)$ and say that $\extriangulatedCat$ \emph{has enough $\extFunc$-injectives} if, for each $C\in\extriangulatedCat$, there exists an $\additiveReal$-inflation $C\to I$ with $I\in\text{Inj}_\extFunc(\extriangulatedCat)$. 

    \item The $\extFunc$-\emph{projective} (respectively, $\extFunc$-\emph{injective}) \emph{dimension} of $\subcat\subseteq\extriangulatedCat$ is
            \begin{align*}
                \text{pd}_\extFunc(\subcat) &:= \text{sup}\big(\big\{\text{min}(\{m\in\mathbb{Z}_{\ge0} \mid S\in \text{Proj}_\extFunc(\extriangulatedCat)^\land_m\}) \mid S\in\subcat\big\}\big). \\
                \bigg( \text{id}_\extFunc(\subcat) &:= \text{sup}\big(\big\{\text{min}(\{m\in\mathbb{Z}_{\ge0} \mid S\in \text{Inj}_\extFunc(\extriangulatedCat)^\lor_m\}) \mid S\in\subcat\big\}\big).  \bigg)
            \end{align*}
    \end{enumerate}
\end{Prp-Dfn}

\begin{Dfn}\label{dfn: cotorsion}
    A \emph{cotorsion pair} in $\extriangulatedCat$ is a pair $\cotorsionPair$ of subcategories of $\extriangulatedCat$ such that $\cotorsionClass={}^{\perp_1}\cotorsionfreeClass$ and $\cotorsionClass^{\perp_1}=\cotorsionfreeClass$. Furthermore, it is
    \begin{enumerate}[label=\DfnLbl]
        \item \emph{generated by} $\subcat\subseteq\extriangulatedCat$ if $\cotorsionfreeClass=\subcat^{\perp_1}$;
        \item \emph{cogenerated by} $\subcat\subseteq\extriangulatedCat$ if $\cotorsionClass={}^{\perp_1}\subcat$;
        \item \emph{complete} if, for each $C\in\extriangulatedCat$, there exist $\additiveReal$-conflations $Y\to X\to C$ and $C\to Y'\to X'$ with $X,X'\in\cotorsionClass$ and $Y,Y'\in\cotorsionfreeClass$;
        \item \emph{bounded} if $\cotorsionClass^\land=\extriangulatedCat=\cotorsionfreeClass^\lor$.
    \end{enumerate}
\end{Dfn}

\begin{Rmk}\label{rmk: orthogonality}
    Taking orthogonals reverses inclusions, i.e., if $\subcat_1\subseteq\subcat_2\subseteq\extriangulatedCat$, then ${}^{\perp_1}\subcat_2\subseteq{}^{\perp_1}\subcat_1$ and $\subcat_2^{\perp_1}\subseteq\subcat_1^{\perp_1}$. From this it follows that any cotorsion pair $\cotorsionPair$ in $\extriangulatedCat$ is determined uniquely by either $\cotorsionClass$ or $\cotorsionfreeClass$, and moreover that $\text{Proj}_\extFunc(\extriangulatedCat)\subseteq \cotorsionClass$ and $\text{Inj}_\extFunc(\extriangulatedCat)\subseteq \cotorsionfreeClass$.
\end{Rmk}

\begin{Dfn}\label{dfn: positive}\cite[\S3]{gorsky2021positive}
    We say that $\extriangulatedTriple$ \emph{has positive extensions} if, for each $n>1$, there exists an additive bifunctor $\extFunc^n:\mathscr{C}^\text{op}\times \mathscr{C}\to \text{Ab}$ such that, for each $\additiveReal$-conflation $A\to B\to C$ and $X\in\extriangulatedCat$, there exist exact sequences in Ab of the form:
    \begin{align*}
        \extFunc^{n-1}(X,C)\to \extFunc^n(X,A)\to \extFunc^n(X,B)\to \extFunc^n(X,C), \\
        \extFunc^{n-1}(A,X)\to \extFunc^n(C,X)\to \extFunc^n(B,X)\to \extFunc^n(A,X),
    \end{align*}
    where $\extFunc^1:=\extFunc$. In this case, we define $\extFunc^0:=\text{Hom}_\extriangulatedCat$ (c.f. Proposition~\ref{prp: sequences}). 
\end{Dfn}

\subsection{Positive extensions and silting subcategories}\label{ssec: positive}

Let $\extriangulatedTriple$ have positive extensions. This allows us to consider silting subcategories and relate them to a particular class of cotorsion pairs.

\begin{Dfn}\label{dfn: subcategories'}
    Let $\subcat\subseteq \extriangulatedCat$. 
    \begin{enumerate}[label=\DfnLbl] 

        \item For $I\subseteq \mathbb{Z}_{\ge 0}$, we define the subcategories
            \begin{align*}
                {}^{\perp_I}\subcat &:= \{C\in\extriangulatedCat \mid \extFunc^i(C,S)=0 \text{ for all } S\in\subcat \text{ and } i\in I \}; \\
                \subcat^{\perp_I} &:= \{C\in\extriangulatedCat \mid \extFunc^i(S,C)=0 \text{ for all } S\in\subcat \text{ and } i\in I \};
            \end{align*}
            where $I$ is denoted by an interval (possibly using shorthand notation such as $>0$) if it is such.

        \item We say that $\subcat$ is a \emph{presilting} subcategory of $\extriangulatedCat$ if $\subcat\subseteq {}^{\perp_{>0}}\subcat$ (equivalently, $\subcat\subseteq\subcat^{\perp_{>0}}$) and $\subcat$ is closed under summands in $\extriangulatedCat$. If, furthermore, $\text{thick}(\subcat)=\extriangulatedCat$, then it is a \emph{silting} subcategory of $\extriangulatedCat$.

        \item If $\extriangulatedCat$ has enough $\extFunc$-projectives, we say that a presilting subcategory $\subcat$ of $\extriangulatedCat$ is $m$-\emph{tilting} if there exists $m\in\mathbb{Z}_{\ge0}$ such that $\subcat\subseteq\text{Proj}_\extFunc(\extriangulatedCat)^\land_m$ and $\text{Proj}_\extFunc(\extriangulatedCat)\subseteq\subcat^\lor_m$.

        \item If $\extriangulatedCat$ has enough $\extFunc$-injectives, we say that a presilting subcategory $\subcat$ of $\extriangulatedCat$ is $m$-\emph{cotilting} if there exists $m\in\mathbb{Z}_{\ge0}$ such that $\subcat\subseteq\text{Inj}_\extFunc(\extriangulatedCat)^\lor_m$ and $\text{Inj}_\extFunc(\extriangulatedCat)\subseteq\subcat^\land_m$.

        \item A cotorsion pair $\cotorsionPair$ in $\extriangulatedCat$ is \emph{hereditary} if $\cotorsionClass\subseteq {}^{\perp_{>0}}\cotorsionfreeClass$.
    \end{enumerate}
\end{Dfn}

\begin{Rmk}\label{rmk: heredity}
    Let $\subcat\subseteq\extriangulatedCat$.
    \begin{enumerate}[label=\RmkLbl]
        \item Analogously to Remark~\ref{rmk: subcategories}(1), it follows from Definition~\ref{dfn: positive} that $\subcat^{\perp_{>0}}$ (respectivley, ${}^{\perp_{>0}}\subcat$) is closed under extensions, summands and (co)cones in $\extriangulatedCat$. 

        \item As in Remark~\ref{rmk: orthogonality}, if $\subcat_1\subseteq\subcat_2\subseteq\extriangulatedCat$, then ${}^{\perp_{>0}}\subcat_2\subseteq{}^{\perp_{>0}}\subcat_1$ and $\subcat_2^{\perp_{>0}}\subseteq\subcat_1^{\perp_{>0}}$.

        \item If $\cotorsionPair$ is a hereditary cotorsion pair in $\extriangulatedCat$, it follows from (2) that $\cotorsionKernel$ is a presilting subcategory of $\extriangulatedCat$. This hints towards a fundamental relationship between hereditary cotorsion pairs and silting subcategories, which is explicitly stated in Theorem~\ref{thm: AT}.
    \end{enumerate}
\end{Rmk}

The following result is central to this work, as will become evident in Section~\ref{ssec: producers}. It is an extriangulated generalization of \cite[Corollary~5.9]{mendoza2013auslander} (see \cite[Corollary~5.11]{adachi2022hereditary}).

\begin{Thm}\cite[Theorem~5.7]{adachi2022hereditary}\label{thm: AT}
    There exist mutually inverse bijections
    \begin{center}
        \begin{tikzcd}
            \{ \text{bounded hereditary complete cotorsion pairs in } \extriangulatedCat \} \arrow[shift left]{r}[]{\Phi} & \{\text{silting subcategories of } \extriangulatedCat \} \arrow[shift left]{l}[]{\Psi},
        \end{tikzcd}
    \end{center}
    where $\Phi\cotorsionPair:=\cotorsionKernel$ and $\Psi(\subcat):=(\subcat^\lor,\subcat^\land)$. 
\end{Thm}

Theorem~\ref{thm: AT} induces a partial order for silting subcategories of $\extriangulatedCat$ as follows.

\begin{Prp-Dfn}\cite[Proposition~5.12]{adachi2022hereditary}\label{prp-dfn: order}
    For silting subcategories $\mathcal{M}$ and $\mathcal{N}$ of $\extriangulatedCat$, let $\mathcal{M}\geqslant\mathcal{N}$ denote the condition $\extFunc^n(\mathcal{M},\mathcal{N})=0$ for all $n>0$. Then the following statements are equivalent.
    \begin{enumerate}[label=\PrpLbl]
        \item $\mathcal{M}\geqslant\mathcal{N}$.
        \item $\mathcal{M}^\land\supseteq\mathcal{N}^\land$.
        \item $\mathcal{M}^\lor\subseteq\mathcal{N}^\lor$.
    \end{enumerate}
    In particular, $\geqslant$ is a partial order on the collection of silting subcategories of $\extriangulatedCat$.
\end{Prp-Dfn}

Another important fact about silting subcategories $\subcat$ is that they are maximal with respect to the orthogonality condition $\subcat\subseteq{}^{\perp_{>0}}\subcat$.

\begin{Lmm}\cite[Lemma~5.3]{adachi2022hereditary}\label{lmm: maximal}
    Let $\mathcal{M}$ be a silting subcategory of $\extriangulatedCat$. If $\mathcal{N}\subseteq\extriangulatedCat$ is such that $\mathcal{M}\subseteq\mathcal{N}\subseteq{}^{\perp_{>0}}\mathcal{N}$, then $\mathcal{M}=\mathcal{N}$.
\end{Lmm}

The next result is fundamental for understanding the mutation of silting subcategories in extriangulated categories, so we include a proof for completeness. For details, we refer to \cite[\S4.2]{adachi2025assortment} (c.f. \cite[\S2.4]{aihara2012silting} for the original treatment in triangulated categories).

\begin{Dfn}\label{dfn: approximations}
    Let $\mathcal{D}\subseteq\subcat\subseteq\extriangulatedCat$.
    \begin{enumerate}[label=\DfnLbl]
        \item A \emph{left} $\mathcal{D}$-\emph{approximation} (or $\mathcal{D}$-\emph{preenvelope}) of an object $C\in\extriangulatedCat$ is a morphism $C\xrightarrow[]{d}D$ in $\extriangulatedCat$ with $D\in\mathcal{D}$ such that $\text{Hom}_\mathscr{C}(d,D'):\text{Hom}_\mathscr{C}(D,D')\to \text{Hom}_\mathscr{C}(C,D')$ is surjective for all $D'\in\mathcal{D}$. A \emph{right} $\mathcal{D}$-\emph{approximation} (or $\mathcal{D}$-\emph{precover}) of an object $C\in\extriangulatedCat$ is defined dually.
        \item We say that $\mathcal{D}$ is a \emph{covariantly finite subcategory} of $\subcat$ if each object in $\subcat$ has a left $\mathcal{D}$-approximation. If, furthermore, each object in $\subcat$ has a left $\mathcal{D}$-approximation which is an $\additiveReal$-inflation, then it is a \emph{good} covariantly finite subcategory of $\subcat$. A \emph{(good) contravariantly finite} subcategory of $\subcat$ is defined dually.
    \end{enumerate}
\end{Dfn}

\begin{Lmm}\cite[Lemma~4.8]{adachi2025assortment}\label{lmm: good}
    Let $\mathcal{D}\subseteq\subcat\subseteq\extriangulatedCat$ with $\subcat\subseteq{}^{\perp_{>0}}\subcat$. Suppose that there exists an $\additiveReal$-conflation $S\xrightarrow[]{f}D\xrightarrow[]{g}N$ with $S\in\subcat$ and $f$ a left $\mathcal{D}$-approximation of $S$. Then
    \begin{itemize}
        \item $f$ is a left $\mathcal{D}$-approximation of $S$ with $N\in{}^{\perp_{>0}}\mathcal{D}$;
        \item $g$ is a right $\mathcal{D}$-approximation of $N$ with $S\in\mathcal{D}^{\perp_{>0}}$;
        \item $N\in\subcat^{\perp_{>0}}\cap{}^{\perp_{>1}}\subcat$.
    \end{itemize}
\end{Lmm}
\begin{proof}
    Note that applying $\text{Hom}_\extriangulatedCat(-,\mathcal{D})$ to the $\additiveReal$-conflation induces the exact sequence
    \[
        \extFunc^n(D,\mathcal{D})\to \extFunc^n(S,\mathcal{D})\to \extFunc^{n+1}(N,\mathcal{D})\to \extFunc^{n+1}(D,\mathcal{D})
    \] 
    for all $n\ge0$, and that $\mathcal{D}\subseteq\subcat\subseteq{}^{\perp_{>0}}\subcat\subseteq{}^{\perp_{>0}}\mathcal{D}$, since taking orthogonals reverses inclusions. The second and fourth terms then vanish for all $n>0$, which implies that $N\in{}^{\perp_{>1}}\mathcal{D}$. For $n=0$, the fourth term also vanishes and the leftmost morphism $\text{Hom}_\extriangulatedCat(f,\mathcal{D}):\text{Hom}_\extriangulatedCat(D,\mathcal{D})\to \text{Hom}_\extriangulatedCat(S,\mathcal{D})$ is surjective because $f$ is a left $\mathcal{D}$-approximation of $S$, from which it follows that $N\in{}^{\perp_1}\mathcal{D}$. Hence, $N\in{}^{\perp_{>0}}\mathcal{D}$. 

    Similarly, applying $\text{Hom}_\extriangulatedCat(\mathcal{D},-)$ to the $\additiveReal$-conflation induces the exact sequence
    \[
        \text{Hom}_\extriangulatedCat(\mathcal{D},D)\xrightarrow[]{\text{Hom}_\extriangulatedCat(\mathcal{D},g)} \text{Hom}_\extriangulatedCat(\mathcal{D},N) \to \extFunc(\mathcal{D},S),
    \] 
    and moreover $\subcat\subseteq \subcat^{\perp_{>0}}\subseteq\mathcal{D}^{\perp_{>0}}$. The rightmost term then vanishes, by which the leftmost morphism is a surjection. Thus, $g$ is a right $\mathcal{D}$-approximation of $N$ and, by analogous arguments to the previous paragraph, $S\in\mathcal{D}^{\perp_{>0}}$.

    On the other hand, since $\subcat^{\perp_{>0}}$ is closed under cones in $\extriangulatedCat$ and $\mathcal{D}\subseteq\subcat\subseteq\subcat^{\perp_{>0}}$, it follows from the $\additiveReal$-conflation that $N\in\subcat^{\perp_{>0}}$. Finally, by replacing $\text{Hom}_\extriangulatedCat(-,\mathcal{D})$ with $\text{Hom}_\extriangulatedCat(-,\subcat)$ in the first paragraph, one obtains $N\in{}^{\perp_{>1}}\subcat$.
\end{proof}

\subsection{(Co)products}\label{ssec: (co)products}

Large silting mutation diverges into two possibilities, since finite direct sums can be generalized as either infinite products or infinite coproducts. For reasons explained in Section~\ref{sec: large}, we will focus our preliminaries on products. For a more in-depth discussion on products and coproducts in extriangulated categories, we refer to \cite[]{argudín2025recollements}.

\begin{Dfn}\label{dfn: exact products}\cite[Definition~3.2~\&~dual of Proposition~3.4]{argudín2025recollements}
    We say that $\extriangulatedTriple$ 
    \begin{enumerate}[label=\DfnLbl]
        \item \emph{has products} (or \emph{is (AET3*)}) if, for each family $\{C_j\}_{j\in J}$ of objects in $\extriangulatedCat$, the product
            \[
                \bigg(C=\prod_{j\in J}C_j, \big\{\pi^C_j:C\to C_j\big\}_{j\in J}\bigg)
            \] 
            exists in $\extriangulatedCat$;
        \item \emph{has exact products} (or \emph{is (AET3.5*)}) if it has products and, for all families $\{A_j\}_{j\in J}$, $\{C_j\}_{j\in J}$ of objects in $\extriangulatedCat$, there is a natural transformation 
            \[
                \Gamma^*: \prod_{j\in J}\extFunc(C_j, A_j)\to \extFunc\bigg(\prod_{j\in J}C_j, \prod_{j\in J}A_j\bigg) 
            \]
            such that, if $A_i\xrightarrow[]{f_i} E_i\xrightarrow[]{g_i} C_i$ is an $\additiveReal$-conflation corresponding to $\eta_i\in\extFunc(C_i,A_i)$ for each $i\in J$, then
            \[
                \prod_{j\in J}A_j \xrightarrow[]{\prod_{j\in J}f_j} \prod_{j\in J}E_j \xrightarrow[]{\prod_{j\in J}g_j} \prod_{j\in J}C_j
            \] 
            is an $\additiveReal$-conflation corresponding to $\Gamma^\ast\big(\prod_{j\in J}\eta_j\big)\in \extFunc\big(\prod_{j\in J}C_j, \prod_{j\in J}A_j\big)$.
    \end{enumerate}
\end{Dfn}

\begin{Dfn}\label{dfn: products}
    Let $\extriangulatedTriple$ have products.
    \begin{enumerate}[label=\DfnLbl] 
        \item We say that $\subcat\subseteq\extriangulatedCat$ is \emph{closed under products} in $\extriangulatedCat$ if, for each family $\{S_j\}_{j\in J}$ of objects in $\subcat$, we have that $\prod_{j\in J}S_j\in\subcat$.
        \item We denote the closure of $\subcat\subseteq\extriangulatedCat$ under products and summands in $\extriangulatedCat$ by $\text{Prod}(\subcat)$. 
        \item We say that $Q\in\text{Inj}_\extFunc(\extriangulatedCat)$ is an $\extFunc$-\emph{injective cogenerator} of $\extriangulatedCat$ if, for each $C\in\extriangulatedCat$, there exist a set $J$ and an $\additiveReal$-inflation $C\to Q^J$.
    \end{enumerate}
    If $(\extriangulatedCat^\text{op},\extFunc^\text{op},\additiveReal^\text{op})$ has products, the dual notions of (a) and (c) are defined analogously in $\extriangulatedCat$, and the closure of $\subcat\subseteq\extriangulatedCat$ under coproducts and summands in $\extriangulatedCat$ is denoted by $\text{Add}(\subcat)$.
\end{Dfn}

\begin{Rmk}\label{rmk: products}
    Let $\extriangulatedTriple$ have products.
    \begin{enumerate}[label=\RmkLbl]
        \item By the dual of \cite[Proposition~3.1(b)]{argudín2025recollements} we have that, for each $X\in\extriangulatedCat$ and family $\{C_j\}_{j\in J}$ of objects in $\extriangulatedCat$, there exists a monomorphism
            \[
                \extFunc\bigg(X,\prod_{j\in J}C_j\bigg) \hookrightarrow \prod_{j\in J} \extFunc(X,C_j).
            \] 
            Since $\extFunc$ is additive, for each $\subcat\subseteq\extriangulatedCat$, it follows that $\text{Prod}(\subcat^{\perp_1})=\subcat^{\perp_1}$. In particular, $\text{Prod}(\text{Inj}_\extFunc(\extriangulatedCat))=\text{Inj}_\extFunc(\extriangulatedCat)$.

        \item If $\extriangulatedCat$ has an $\extFunc$-injective cogenerator $Q$, then (1) implies that $\text{Prod}(Q)\subseteq\text{Inj}_\extFunc(\extriangulatedCat)$, by which $\extriangulatedCat$ has enough $\extFunc$-injectives, and moreover that $\text{Prod}(Q)=\text{Inj}_\extFunc(\extriangulatedCat)$, since all $\additiveReal$-conflations of the form $I\to Q^J\to C$ with $I\in\text{Inj}_\extFunc(\extriangulatedCat)$ split.
    \end{enumerate}
\end{Rmk}

\section{Large silting mutation}\label{sec: large}

Our main motivation for this work lies in cosilting mutation (see \cite[]{hügel2025mutation} for reference), a topic inextricably linked to injective modules, which are in general closed under (direct) products but not coproducts. We thus develop a theory of mutation for product-closed silting subcategories in Section~\ref{ssec: product-closed}, noting that the dual theory for coproduct-closed silting subcategories is valid for extriangulated categories satisfying the dual conditions. We also focus on right approximations instead of left ones (for readers familiar with $\tau$-tilting theory: this is because cosilting modules\textemdash i.e., the $0$-th cohomologies of 2-term cosilting complexes\textemdash are a large generalization of support $\tau^{-1}$-tilting modules). We then show in Section~\ref{ssec: producers} how, under some additional hypotheses, this theory can be used to mutate the objects which produce product-closed silting subcategories, which we aptly call \emph{producers}. Section~\ref{ssec: special} focuses on some special cases which appear in the examples of Section~\ref{sec: examples} and can be given a more detailed description. 

\subsection{Product-closed silting subcategories}\label{ssec: product-closed}

We develop a theory of mutation for product-closed silting subcategories in extriangulated categories. Let $\extriangulatedTriple$ have positive extensions and exact products, and $\subcat=\text{Prod}(\subcat)$ be a presilting subcategory of $\extriangulatedCat$.

\begin{Dfn}\label{dfn: product-mutation}
    Let $\mathcal{D}$ be a good contravariantly finite subcategory of $\subcat$ and, for each $S\in\subcat$, consider a right $\mathcal{D}$-approximation $D\xrightarrow[]{g}S$ and its corresponding $\additiveReal$-conflation 
    \[
    N_S\to D\xrightarrow[]{g} S.
    \] 
    We define the \emph{right product-mutation} of the presilting subcategory $\subcat$ with respect to $\mathcal{D}$ as
    \[
        \prod\nolimits^R(\subcat;\mathcal{D}):=\text{Prod}(\mathcal{D}\cup\{N_S\}_{S\in\subcat}).
    \] 
    The \emph{left product-mutation} $\prod\nolimits^L(\subcat;\mathcal{D})$ of $\subcat$ with respect to a good \emph{covariantly} finite subcategory $\mathcal{D}$ of $\subcat$ is defined dually.
\end{Dfn}

\begin{Rmk}\label{rmk: product-mutation}
    Following the argument in \cite[Remark~4.10(1)]{adachi2025assortment}, the left (respectively, right) product-mutation of $\subcat$ with respect to a good co(ntra)variantly finite subcategory $\mathcal{D}$ does not depend on a particular choice of a left (right) $\mathcal{D}$-approximation. The same argument appears explicitly in the proof of Proposition~\ref{prp: left invariance}, which implies that, when considering left product-mutation, we may assume that the good covariantly finite subcategory with respect to which we mutate is product-closed. 

    We note that the aforementioned implication follows automatically if $\extriangulatedTriple$ is \emph{weakly idempotent complete}, i.e., if for each composition of morphisms $h=gf$ in $\extriangulatedCat$ such that $h$ is an $\additiveReal$-inflation, the same follows for $f$, and dually for $\additiveReal$-deflations. In this case, one can prove that a subcategory of $\extriangulatedCat$ is closed under summands if, and only if, it is closed under retracts, and then use the universal property of the product to prove that any covariantly finite subcategory of $\extriangulatedCat$ which is closed under summands in $\extriangulatedCat$ is also closed under products in $\extriangulatedCat$. It then suffices to observe that, by \cite[Remark~2.18]{nakaoka2019extriangulated}, weak idempotent completeness is inhereted by subcategories of $\extriangulatedCat$ which are closed under extensions in $\extriangulatedCat$. Thus, if $\subcat$ is a silting subcategory of $\extriangulatedCat$, then $\text{Prod}(\subcat)=\subcat$.
\end{Rmk}

\begin{Prp}[]\label{prp: left invariance}
    Let $\mathcal{D}$ be a good covariantly finite subcategory of $\subcat$. Then $\text{Prod}(\mathcal{D})$ is a good covariantly finite subcategory of $\subcat$ and
    \[
        \prod\nolimits^L(\subcat;\mathcal{D})=\prod\nolimits^L(\subcat;\text{Prod}(\mathcal{D})).
    \] 
\end{Prp}
\begin{proof}
    It follows from the universal property of the product that any left $\mathcal{D}$-approximation of an object is a left $\text{Prod}(\mathcal{D})$-approximation of that object. Moreover, since $\subcat=\text{Prod}(\subcat)$ by assumption, $\text{Prod}(\mathcal{D})$ is a good covariantly finite subcategory of $\subcat$. In particular, we have that $\prod\nolimits^L(\subcat;\mathcal{D})\subseteq \prod\nolimits^L(\subcat;\text{Prod}(\mathcal{D}))$. To prove the opposite inclusion, consider an $\additiveReal$-conflation $S\xrightarrow[]{f'}D'\to T'$ such that $f'$ is a left $\text{Prod}(\mathcal{D})$-approximation of an object $S\in\subcat$. By definition, $\prod\nolimits^L(\subcat;\mathcal{D})$ contains $\text{Prod}(\mathcal{D})$, so it suffices to show that $T'\in \prod\nolimits^L(\subcat;\mathcal{D})$. Since $\mathcal{D}$ is a good covariantly finite subcategory of $\subcat$, there exists an $\additiveReal$-conflation $S\xrightarrow[]{f}D\to T$ such that $f$ is a left $\mathcal{D}$-approximation of $S$. By \cite[Proposition~3.15]{nakaoka2019extriangulated}, we have the commutative diagram of $\additiveReal$-conflations
    \begin{center}
        \begin{tikzcd}
            S \arrow[]{d}[swap]{f'} \arrow[]{r}[]{f} &D \arrow[]{d}[]{} \arrow[]{r}[]{} &T \arrow[equals]{d}[]{} \\
            D' \arrow[]{d}[]{} \arrow[]{r}[]{} &E \arrow[]{d}[]{} \arrow[]{r}[]{} &T. \\
            T' \arrow[equals]{r}[]{} &T'
        \end{tikzcd}
    \end{center}
    Note that $T\in{}^{\perp_{>0}}\mathcal{D}\subseteq {}^{\perp_1}\mathcal{D}$ and $T'\in{}^{\perp_1}\text{Prod}(\mathcal{D})$ by Lemma~\ref{lmm: good}. Moreover, ${}^{\perp_1}\mathcal{D}={}^{\perp_1}\text{Prod}(\mathcal{D})$ by Remark~\ref{rmk: products}(1) and the fact that taking orthogonals reverses inclusions. Thus, both $\additiveReal$-conflations with middle term $E$ split, by which $T\bigoplus D'\simeq T'\bigoplus D$. Since $D',T\in\prod\nolimits^L(\subcat;\mathcal{D})$, which is closed under products and summands in $\extriangulatedCat$, $T'\in\prod\nolimits^L(\subcat;\mathcal{D})$.
\end{proof}

Regarding right product-mutation, we cannot assume in general that the contravariantly finite subcategories with respect to which we perform the mutation are product-closed, although the following condition is sufficient.

\begin{Prp}[]\label{prp: right invariance}
    Let $\mathcal{D}$ be a good contravariantly finite subcategory of $\subcat$ such that $\mathcal{D}^{\perp_1}=\text{Prod}(\mathcal{D})^{\perp_1}$. Then $\text{Prod}(\mathcal{D})$ is a good contravariantly finite subcategory of $\subcat$ and
    \[
        \prod\nolimits^R(\subcat;\mathcal{D})=\prod\nolimits^R(\subcat;\text{Prod}(\mathcal{D})).
    \] 
\end{Prp}
\begin{proof}
    Let $S\in\subcat$. By hypothesis, there exists an $\additiveReal$-conflation $N\to D\xrightarrow[]{g}S$, where $g$ is a right $\mathcal{D}$-approximation of $S$. Now, let $D'\in\text{Prod}(\mathcal{D})$. By applying $\text{Hom}_\extriangulatedCat(D',-)$ to the previous $\additiveReal$-conflation we obtain the exact sequence
    \[
    \text{Hom}_\extriangulatedCat(D',N)\to \text{Hom}_\extriangulatedCat(D',D)\xrightarrow[]{\text{Hom}_\extriangulatedCat(D',g)} \text{Hom}_\extriangulatedCat(D',S)\to \extFunc(D',N).
    \] 
    By the dual of Lemma~\ref{lmm: good}, we know that $N\in\mathcal{D}^{\perp_{>0}}\subseteq\mathcal{D}^{\perp_1}$. In particular, $N\in\text{Prod}(\mathcal{D})^{\perp_1}$ by our hypothesis, which implies that the rightmost term in the previous sequence vanishes, by which $\text{Hom}_\extriangulatedCat(D',g)$ is surjective. Thus, $g$ is a right $\text{Prod}(\mathcal{D})$-approximation of $S$. Since $\subcat=\text{Prod}(\subcat)$, it follows that $\text{Prod}(\mathcal{D})$ is a good contravariantly finite subcategory of $\subcat$, and in particular we have that $\prod\nolimits^R(\subcat;\mathcal{D})\subseteq \prod\nolimits^R(\subcat;\text{Prod}(\mathcal{D}))$. To prove the opposite inclusion, we proceed analogously to the proof of Proposition~\ref{prp: left invariance}, using once again (the dual of) Lemma~\ref{lmm: good} and the hypothesis that $\mathcal{D}^{\perp_1}=\text{Prod}(\mathcal{D})^{\perp_1}$.
\end{proof}

\noindent We note that this and other asymmetries arising in large silting mutation often stem from the fact that, in general, families of left and of right approximations behave differently under taking products, as well as under taking coproducts.

We are now ready to prove the main results of this section, which are dual large analogues of \cite[Lemma~4.11,~Theorem~4.12~\&~Proposition~4.13]{adachi2025assortment}. Although the proofs are the same in spirit, we include them here for completeness.

\begin{Lmm}[]\label{lmm: product-mutation}
    Let $\mathcal{D}$ be a good contravariantly finite subcategory of $\subcat$ such that $\mathcal{D}^{\perp_1}=\text{Prod}(D)^{\perp_1}$. Then, for each $N\in\prod\nolimits^R(\subcat;\mathcal{D})$, there exists an $\additiveReal$-conflation $N\xrightarrow[]{f_N}D_N\xrightarrow[]{g_N}S_N$ such that $S_N\in\subcat$, $g_N$ is a right $\text{Prod}(\mathcal{D})$-approximation of $S_N$ and $f_N$ is a left $\text{Prod}(\mathcal{D})$-approximation of $N$. In particular, $\text{Prod}(\mathcal{D})$ is a good covariantly finite subcategory of $\prod\nolimits^R(\subcat;\mathcal{D})$.
\end{Lmm}
\begin{proof}
    Let $N\in\prod\nolimits^R(\subcat;\mathcal{D})$. By our hypothesis and Proposition~\ref{prp: right invariance}, $N\in\prod\nolimits^R(\subcat;\text{Prod}(\mathcal{D}))$. Thus, $N$ is a summand of some $N'$, which is itself a product of objects in $\mathcal{D}$ and cocones of right $\text{Prod}(\mathcal{D})$-approximations of objects in $\subcat$. Let $\{\mathcal{D}_j\}_{j\in J}\subseteq\mathcal{D}$ and $\{N_{S_k}\xrightarrow[]{f_k}D_k\xrightarrow[]{g_k}S_k\}_{k\in K}$ be a family of $\additiveReal$-conflations such that $N'=\prod_{j\in J}D_j\bigoplus\prod_{k\in K}N_{S_k}$, where $g_k$ is a right $\text{Prod}(\mathcal{D})$-approximation of $S_k\in\subcat$ for each $k\in K$. Since $\extriangulatedTriple$ has exact products, by considering the trivial $\additiveReal$-conflation given by the identity morphism $\text{Id}_{D_j}$ for each $j\in J$, we obtain the $\additiveReal$-conflation 
    \[
        N'\xrightarrow[]{\prod_{j\in J}\text{Id}_{D_j}\oplus \prod_{k\in K}f_k} \prod_{j\in J}D_j\bigoplus\prod_{k\in K}D_k \xrightarrow[]{0^J\oplus\prod_{j\in K}g_k}\prod_{k\in K}S_k,
    \] 
    which we denote by $N'\xrightarrow[]{f_{N'}}D_{N'}\xrightarrow[]{g_{N'}} S_{N'}$. Note that, by assumption, $S_{N'}\in\text{Prod}(\subcat)=\subcat$.

    Consider the split $\additiveReal$-conflation $N\to N'\to N''$, given by the canonical inclusion. By the extriangulated octahedral axiom, there exists a commutative diagram of $\additiveReal$-conflations
    \begin{center}
        \begin{tikzcd}
            N\arrow[equals]{d}[]{}\arrow[]{r}[]{} &N'\arrow[]{d}[]{f_{N'}}\arrow[]{r}[]{} &N''\arrow[]{d}[]{} \\
            N\arrow[]{r}[]{f_N} &D_{N'}\arrow[]{d}[]{g_{N'}}\arrow[]{r}[]{g_N} &S_N \arrow[]{d}[]{} \\
                                            &S_{N'}\arrow[equals]{r}[]{} &S_{N'}.
        \end{tikzcd}
    \end{center}
    Since $N''\in\text{Prod}(\mathcal{D})\subseteq\subcat$, we have that $\extFunc(S_{N'},N'')=0$. Hence, $S_N\simeq N''\bigoplus S_{N'}$, which implies that $S_N\in\subcat$. Now, let $D'\in\text{Prod}(\mathcal{D})$. By applying $\text{Hom}_\extriangulatedCat(D',-)$ to the middle row of the previous diagram, we obtain the exact sequence
    \[
        \text{Hom}_\extriangulatedCat(D',D_{N'})\xrightarrow[]{\text{Hom}_\extriangulatedCat(D',g_N)}\text{Hom}_\extriangulatedCat(D',S_N)\to \extFunc(D',N).
    \] 
    where the rightmost term is a summand of $\extFunc(D',N')$ due to the additivity of $\extFunc(D',-)$. Since $\text{Prod}(\mathcal{D})^{\perp_1}$ is closed under products in $\extriangulatedCat$ by Remark~\ref{rmk: products}(1), it follows that $\extFunc(D',N')$ \textemdash \ hence, $\extFunc(D',N)$ \textemdash \ vanishes, and by the above exact sequence, that $g_N$ is a right $\text{Prod}(\mathcal{D})$-approximation of $S_N$. The dual of Lemma~\ref{lmm: good} then implies that $f_N$ is a left $\text{Prod}(\mathcal{D})$-approximation of $N$.
\end{proof}

\begin{Thm}[]\label{thm: mutation}\leavevmode
    Let $\extriangulatedTriple$ be an extriangulated category with positive extensions and exact products, $\subcat$ be a presilting subcategory of $\extriangulatedCat$ and $\mathcal{D}$ be a good contravariantly finite subcategory of $\subcat$ such that $\text{Prod}(\subcat)=\subcat$ and $\text{Prod}(\mathcal{D})=\mathcal{D}$. Then the following statements hold. 
    \begin{enumerate}[label=\ThmLbl]
        \item $\prod\nolimits^R(\subcat;\mathcal{D})$ is a product-closed presilting subcategory of $\extriangulatedCat$.
        \item $\prod\nolimits^R(\subcat;\mathcal{D})\geqslant\subcat$, where the equality holds if, and only if, $\subcat=\mathcal{D}$.
        \item If $\subcat$ is a silting subcategory of $\extriangulatedCat$, then so is $\prod\nolimits^R(\subcat;\mathcal{D})$.
        \item $\mathcal{D}$ is a good covariantly finite subcategory of $\prod\nolimits^R(\subcat;\mathcal{D})$ and
            \[
                \prod\nolimits^L\big(\prod\nolimits^R(\subcat;\mathcal{D});\mathcal{D}\big)=\subcat.
            \] 
    \end{enumerate}
\end{Thm}
\begin{proof}
    Let $\mathcal{N}:=\prod\nolimits^R(\subcat;\mathcal{D})$.
    \begin{enumerate}[label=\ThmLbl]
    
        \item Consider $N\in\mathcal{N}$. By Lemma~\ref{lmm: product-mutation}, there exists an $\additiveReal$-conflation $N\xrightarrow[]{f_N}D_N\xrightarrow[]{g_N}S_N$ such that $S_N\in\subcat$ and $g_N$ is a right $\mathcal{D}$-approximation of $S_N$. Applying $\text{Hom}_\extriangulatedCat(-,\mathcal{N})$ to the above $\additiveReal$-conflation induces, for each $n>0$, an exact sequence of the form
            \[
                \extFunc^n(D_N,\mathcal{N})\to \extFunc^n(N,\mathcal{N})\to \extFunc^{n+1}(S_N,\mathcal{N});
            \] 
            moreover, since Lemma~\ref{lmm: product-mutation} can be applied to all objects in $\mathcal{N}$, we have that both the leftmost and rightmost terms vanishes by the dual of Lemma~\ref{lmm: good}. Since $\mathcal{N}=\text{Prod}(\mathcal{N})$ by definition, it follows that $\mathcal{N}$ is a product-closed presilting subcategory of $\extriangulatedCat$.

        \item By the dual of Lemma~\ref{lmm: good}, it follows that $\mathcal{N}\geqslant\subcat$. 

            Now, suppose that $\subcat\geqslant\mathcal{N}$. In particular, we have that $\extFunc(\subcat,\mathcal{N})=0$. By hypothesis, for each $S\in\subcat$, there exists an $\additiveReal$-conflation $N_S\to D\to S$ with $D\in\mathcal{D}$ and $N_S\in\mathcal{N}$, which then splits by assumption. Since $\mathcal{D}$ is closed under summands in $\extriangulatedCat$, it follows that $\subcat=\mathcal{D}$.

            On the other hand, suppose that $\subcat=\mathcal{D}$. By Lemma~\ref{lmm: product-mutation}, $\text{Prod}(\mathcal{D})$ is a good covariantly finite subcategory of $\mathcal{N}$, which implies that $\mathcal{N}\subseteq\mathcal{D}^{\perp_{>0}}$ by the dual of Lemma~\ref{lmm: good}. Thus, $\subcat\geqslant\mathcal{N}$ by assumption.

        \item Note that $\mathcal{D}$ being a good contravariantly finite subcategory of $\subcat$ implies that $\subcat\subseteq \mathcal{N}^\land_1\subseteq \text{thick}(\mathcal{N})$, by which $\text{thick}(\subcat)\subseteq \text{thick}(\mathcal{N})$. The assertion then follows from the hypothesis and (a).

        \item Since $\text{Prod}(\mathcal{D})=\mathcal{D}$, the first assertion follows from Lemma~\ref{lmm: product-mutation}. 

            Let $T\in\prod\nolimits^L(\mathcal{N};\mathcal{D})$. By the dual of Lemma~\ref{lmm: product-mutation}, there exists an $\additiveReal$-conflation $N\xrightarrow[]{f}D\to T$ such that $N\in\mathcal{N}$ and $f$ is a left $\mathcal{D}$-approximation of $N$. By (a), we know that $\mathcal{N}$ is a presilting subcategory of $\extriangulatedCat$, which contains $\mathcal{D}$ by definition. It then follows from Lemma~\ref{lmm: good} that $T\in{}^{\perp_{>0}}\mathcal{D}$. Moreover, by Lemma~\ref{lmm: product-mutation}, $N\in\mathcal{N}$ yields an $\additiveReal$-conflation $N\xrightarrow[]{f'}D'\to S$ such that $S\in\subcat$ and $f'$ is a left $\mathcal{D}$-approximation of $N$. Since $\subcat$ is presilting and $\mathcal{D}\subseteq\subcat$, we have that $S\in{}^{\perp_{>0}}\mathcal{D}$. By proceeding as in the proof of Proposition~\ref{prp: left invariance}, we obtain $T\bigoplus D'\simeq S\bigoplus D\in\subcat$, and since $\subcat$ is closed under summands in $\extriangulatedCat$, $T\in\subcat$.

            Now, let $S\in\subcat$. By hypothesis, there exists an $\additiveReal$-conflation $N\xrightarrow[]{f}D\xrightarrow[]{g} S$ such that $g$ is a right $\mathcal{D}$-approximation of $S$, which implies that  $N\in\mathcal{N}$. By the dual of Lemma~\ref{lmm: good}, $f$ is a left $\mathcal{D}$-approximation of $N$. Moreover, by Lemma~\ref{lmm: product-mutation}, there exists an $\additiveReal$-conflation $N\xrightarrow[]{f_N}D_N\to S_N$ such that $f_N$ is a left $\mathcal{D}$-approximation of $N$. Note that $S_N\in\prod\nolimits^L(\mathcal{N};\mathcal{D})$ by definition. Once again, by proceeding as in the proof of Proposition~\ref{prp: left invariance}, we obtain $S\bigoplus D_N\simeq D\bigoplus S_N\in\prod\nolimits^L(\mathcal{N};\mathcal{D})$, which is closed under summands in $\extriangulatedCat$ by definition. Thus, $S\in\prod\nolimits^L(\mathcal{N};\mathcal{D})$.
    \end{enumerate}
\end{proof}

\subsection{Producers}\label{ssec: producers}

We prove a one-to-one correspondence between product-closed silting subcategories $\subcat$ and equivalence classes of objects $C$ which produce them, in the sense that $\subcat=\text{Prod}(C)$ (Section~\ref{sssec: bijections}). We then define an operation of mutation for such objects, showing that it is equivalent to the product-mutation of their associated product-closed silting subcategories (Section~\ref{sssec: mutation}). Let $\extriangulatedTriple$ have exact products, an $\extFunc$-injective cogenerator $Q$, and all objects with finite $\extFunc$-injective dimension.

\subsubsection{Large silting bijections}\label{sssec: bijections}

\begin{Lmm}\label{lmm: producer}
    Let $\cotorsionPair$ be a bounded hereditary complete cotorsion pair in $\extriangulatedCat$ such that $\text{Prod}(\cotorsionKernel)=\cotorsionKernel$, which is a silting subcategory of $\extriangulatedCat$ with $(\cotorsionKernel)^\land=\cotorsionfreeClass$ by Theorem~\ref{thm: AT}. Consider a finite $(\cotorsionKernel)$-resolution of $Q\in\text{Inj}_\extFunc(\extriangulatedCat)\subseteq\cotorsionfreeClass$
    \begin{align*}
        L_1\to &Z_0\to Q, \\
        L_2\to &Z_1\to L_1, \\
        L_3\to &Z_2\to L_2, \\
               &\vdots \\
        0\to &Z_m\to L_m,
    \end{align*}
    and let $K:=\bigoplus_{i=0}^mZ_i$. Then $\cotorsionKernel=\text{Prod}(K)$.
\end{Lmm}
\begin{proof}
    Note that $Q\in\text{Prod}(K)^\land\subseteq \text{thick}(\text{Prod}(K))$ by construction of $K$. Since $\extriangulatedTriple$ has exact products, $Q^J\in\text{thick}(\text{Prod}(K))$ for any set $J$, by which $\text{Prod}(Q)\subseteq \text{thick}(\text{Prod}(K))$. Moreover, since all objects in $\extriangulatedCat$ have finite $\extFunc$-injective dimension, Remark~\ref{rmk: products}(2) implies that
    \begin{align*}
        \extriangulatedCat &= \text{Inj}_\extFunc(\extriangulatedCat)^\lor \\
                           &= \text{Prod}(Q)^\lor \\
                           &\subseteq \text{thick}(\text{Prod}(K)).
    \end{align*}
    Since $\text{Prod}(\cotorsionKernel)=\cotorsionKernel$, we have that $\text{Prod}(K)\subseteq \cotorsionKernel$, by which $\extFunc^n(\text{Prod}(K),\text{Prod}(K))$ vanishes for all $n>0$. Thus, $\text{Prod}(K)$ is a silting subcategory of $\extriangulatedCat$. Moreover, by Lemma~\ref{lmm: maximal}, silting subcategories are maximal with respect to this orthogonality property, which implies that $\text{Prod}(K)=\cotorsionKernel$. 
\end{proof}

\begin{Prp}\label{prp: bijections}
    There exist mutually inverse bijections
    \begin{center}
        \begin{tikzcd}
            \big\{ \text{objects } C \text{ in } \extriangulatedCat \mid \text{Prod}(C) \text{ is a silting subcategory of } \extriangulatedCat \big\}_{/\sim} \arrow[shift right]{d}[]{}  \arrow[phantom, very near start, xshift=-10mm, yshift=1.25mm]{d}[]{\scriptstyle [C]} \arrow[mapsto, shorten=1.5mm, xshift=-10mm, yshift=0mm]{d}[]{} \\
            \big\{ \text{silting subcategories } \subcat \text{ of } \extriangulatedCat \mid \text{Prod}(\subcat)=\subcat \big\} \arrow[shift right]{d}[swap]{\Psi} \arrow[shift right]{u}[swap]{\Theta} \arrow[phantom, very near start, xshift=-10mm, yshift=-1mm]{u}[]{\scriptstyle \text{Prod}(C)} \\
            \big\{ \text{ bounded hereditary complete cotorsion pairs } \cotorsionPair \text{ in } \extriangulatedCat \mid \text{Prod}(\cotorsionKernel)=\cotorsionKernel \big\} \arrow[shift right]{u}[swap]{\Phi},
        \end{tikzcd}
    \end{center}
    where
    \begin{align*}
        C\sim C' & \ \Leftrightarrow \text{Prod}(C)=\text{Prod}(C'); \\
        \Psi(\subcat) &:=(\subcat^\lor,\subcat^\land); \\
        \Phi\cotorsionPair &:=\cotorsionKernel.
    \end{align*}
    Finally, given a silting subcategory $\subcat$ as in the middle row, apply $\Psi$, consider any finite $\subcat$-resolution of $Q$
    \begin{align*}
        L_1\to &Z_0\to Q, \\
        L_2\to &Z_1\to L_1, \\
        L_3\to &Z_2\to L_2, \\
               &\vdots \\
        0\to &Z_m\to L_m,
    \end{align*}
    and set
    \[
        \Theta\cotorsionPair:=\bigg[\bigoplus_{i=0}^mZ_i\bigg].
    \] 
\end{Prp}
\begin{proof}
    First, note that the bijections in Theorem~\ref{thm: AT} trivially restrict to
    \begin{center}
        \adjustbox{scale=1.22}{\begin{tikzcd}
        \bigg\{ \substack{\text{ bounded hereditary } \\ \text{ complete cotorsion } \\ \text{ pairs } \cotorsionPair \text{ in } \extriangulatedCat } \bigg| \substack{\text{ Prod}(\cotorsionKernel)=\cotorsionKernel } \bigg\} \arrow[shift left]{r}[]{\Phi} & \big\{ \substack{\text{ silting subcategories } \\ \subcat \text{ of } \extriangulatedCat } \mid \substack{\text{Prod}(\subcat)=\subcat } \big\} \arrow[shift left]{l}[]{\Psi}.
        \end{tikzcd}}
    \end{center}
    Now, consider a silting subcategory $\subcat$ of $\extriangulatedCat$ such that $\text{Prod}(\subcat)=\subcat$ and its corresponding bounded hereditary complete cotorsion pair $\Psi(\subcat)=(\subcat^\lor,\subcat^\land)$. Since $\subcat^\lor\cap\subcat^\land=\Phi(\subcat^\lor,\subcat^\land)=\subcat$, it follows from Lemma~\ref{lmm: producer} that the object $K:=\bigoplus_{i=0}^mZ_i$ calculated as in the above statement satisfies $\text{Prod}(K)=\subcat$. This establishes well-defined mutually inverse bijections
    \begin{center}
        \adjustbox{scale=1.2}{\begin{tikzcd}
                \big\{ \substack{\text{ silting subcategories } \\ \subcat \text{ of } \extriangulatedCat } \mid \substack{\text{Prod}(\subcat)=\subcat } \big\} \arrow[yshift=1.25mm]{rr}[]{\Theta} && \big\{ \substack{\text{ objects } C \text{ in } \extriangulatedCat } \mid \substack{ \text{Prod}(C) \text{ is a silting} \\ \text{ subcategory of } \extriangulatedCat} \big\} \arrow[mapsto, shorten=6mm, yshift=-1.25mm, xshift=2mm]{ll}[]{} \arrow[phantom, very near start, yshift=-1.25mm, xshift=1mm]{ll}[]{\scriptstyle [C]} \arrow[phantom, very near end, yshift=-1.25mm, xshift=1mm]{ll}[]{\scriptstyle \text{Prod}(C)}.
        \end{tikzcd}}
    \end{center}
\end{proof}

The following results allow us to rephrase the cotorsion pairs from Proposition~\ref{prp: bijections} as those which appear in Theorem~\ref{thm: B}.

\begin{Lmm}\label{lmm: equivalence}
    Let $\cotorsionPair$ be a bounded hereditary complete cotorsion pair in $\extriangulatedCat$. Then the following conditions are equivalent.
    \begin{enumerate}[label=\LmmLbl]
        \item $\text{Prod}(\cotorsionClass)=\cotorsionClass$.
        \item $\text{Prod}(\cotorsionKernel)=\cotorsionKernel$.
    \end{enumerate}
\end{Lmm}
\begin{proof}
    Recall that $\cotorsionClass={}^{\perp_1}\cotorsionfreeClass$ and $\cotorsionClass^{\perp_1}=\cotorsionfreeClass$ are closed under summands in $\extriangulatedCat$. The implication (a)$\Rightarrow$(b) then follows from the fact that $\text{Prod}(\cotorsionfreeClass)=\cotorsionfreeClass$ by Remark~\ref{rmk: products}(1). On the other hand, suppose that $\text{Prod}(\cotorsionKernel)=\cotorsionKernel$ and let $\{X_j\}_{j\in J}$ be a family of objects in $\cotorsionClass$. By Theorem~\ref{thm: AT}, we know that $\cotorsionClass=(\cotorsionKernel)^\lor$. Moreover, since $\extriangulatedTriple$ has exact products, by considering a finite $(\cotorsionKernel)$-coresolution for each $X_j$, it follows that $\prod_{j\in J}X_j\in(\cotorsionKernel)^\lor=\cotorsionClass$. 
\end{proof}

\begin{Lmm}\label{lmm: bounded}
    Let $\cotorsionPair$ be a cotorsion pair in $\extriangulatedCat$. Then it is bounded if, and only if, $\cotorsionClass^\land=\extriangulatedCat$. In particular, each hereditary complete cotorsion pair $\cotorsionPair$ in $\extriangulatedCat$ such that $\cotorsionClass^\land=\extriangulatedCat$ is uniquely determined by its intersection $\cotorsionKernel$.
\end{Lmm}
\begin{proof}
    Recall that $\text{Inj}_\extFunc(\extriangulatedCat)\subseteq\cotorsionfreeClass$ by definition of $\text{Inj}_\extFunc(\extriangulatedCat)$, since taking orthogonals reverses inclusions. Thus, the first assertion follows from the assumption that all objects in $\extriangulatedCat$ have finite $\extFunc$-injective dimension. The second assertion is then given by Theorem~\ref{thm: AT}.
\end{proof}

In upcoming sections we will work with the class of objects appearing in Proposition~\ref{prp: bijections} and its analogue in the dual setting, for which we introduce some terminology.

\begin{Dfn}\label{dfn: producer}
    Let $\subcat$ be a presilting subcategory of $\extriangulatedCat$. A \emph{(co)producer} of $\subcat$ is an object $C\in\extriangulatedCat$ such that $\text{Prod}(C)=\subcat$ (respectively, $\text{Add}(C)=\subcat$). 
\end{Dfn}

\subsubsection{Mutation of producers via product-closed silting subcategories}\label{sssec: mutation}

We now show how the mutation for product-closed silting subcategories given by Theorem~\ref{thm: mutation} and its analog for left product-mutation induce an operation of mutation on their producers via Proposition~\ref{prp: bijections}.

We begin by considering the following property, for which we give some examples.

\begin{Cdt}\label{cdt: mono}
    Suppose that, for each $n>0$, $X\in\extriangulatedCat$, and family $\{C_j\}_{j\in J}$ of objects in $\extriangulatedCat$, there exists a monomorphism
    \[
        \extFunc^n\bigg(X, \prod_{j\in J}C_j\bigg) \hookrightarrow \prod_{j\in J} \extFunc^n(X,C_J).
    \] 
\end{Cdt}

\begin{Exp}\label{exp: mono}\leavevmode
    It will follow from Proposition-Definition~\ref{prp-dfn: positive}(1) that, if $\extriangulatedCat$ has enough $\extFunc$-projectives, then it satisfies Condition~\ref{cdt: mono} by Remark~\ref{rmk: products}(1).
\end{Exp}

\begin{Lmm}[]\label{lmm: mutation}
    Suppose that Condition~\ref{cdt: mono} is satisfied. Let $C\in\extriangulatedCat$ be a producer of a silting subcategory of $\extriangulatedCat$, $\mathcal{D}\subseteq\extriangulatedCat$ be such that $\text{Prod}(\mathcal{D})\subseteq \text{Prod}(C)$, and $A_1\to A_0\xrightarrow[]{g}C$ be an $\additiveReal$-conflation such that $g$ is a right $\text{Prod}(\mathcal{D})$-approximation of $C$. Then $\text{Prod}(A_1\bigoplus A_0)$ is a silting subcategory of $\extriangulatedCat$.
\end{Lmm}
\begin{proof}
    Clearly, $C\in\text{thick}(A_1\bigoplus A_0)$ by the $\additiveReal$-conflation $A_1\to A_0\xrightarrow[]{g}C$, which implies that $\text{Prod}(C)\subseteq \text{thick}(\text{Prod}(A_1\bigoplus A_0))$ since $\extriangulatedTriple$ has exact products. Moreover, since $\text{Prod}(C)$ is a silting subcategory of $\extriangulatedCat$, it follows that $\text{thick}(\text{Prod}(A_1\bigoplus A_0))=\extriangulatedCat$.

    Consider the $\additiveReal$-conflation $A_1\bigoplus A_0\to A_0\bigoplus A_0\xrightarrow[]{(\begin{smallmatrix} g &0 \end{smallmatrix})} C$. Let $D\in\text{Prod}(\mathcal{D})$. By applying the functor $\text{Hom}_{\extriangulatedCat}(D,-)$ to this $\additiveReal$-conflation we obtain the exact sequence 
    \[
        \text{Hom}_{\extriangulatedCat}\big(D, A_0\bigoplus A_0\big)\to \text{Hom}_{\extriangulatedCat}(D,C)\to \extFunc\big(D, A_1\bigoplus A_0\big)\to \extFunc\big(D, A_0\bigoplus A_0\big).
    \] 
    Note that $A_0\bigoplus A_0\xrightarrow[]{(\begin{smallmatrix} g &0 \end{smallmatrix})}C$ is a right $\text{Prod}(\mathcal{D})$ approximation of $C$, by which the leftmost morphism is surjective. Furthermore, since $A_0\in\text{Prod}(\mathcal{D})\subseteq \text{Prod}(C)\subseteq\text{Prod}(C)^{\perp_{>0}}$, the rightmost term vanishes. Thus, it follows that $\text{Prod}(\mathcal{D})\subseteq {}^{\perp_1}(A_1\bigoplus A_0)$.
    
        Since $\extriangulatedTriple$ has exact products we may consider, for any set $J$, an $\additiveReal$-conflation of the form $\big(A_1\bigoplus A_0\big)^J\to \big(A_0\bigoplus A_0\big)^J\to C^J$. By applying $\text{Hom}_{\extriangulatedCat}(-, A_1\bigoplus A_0)$ to it, we obtain the exact sequence
    \[
    \extFunc^n\big(\big(A_0\bigoplus A_0\big)^J, A_1\bigoplus A_0\big)\to \extFunc^n\big(\big(A_1\bigoplus A_0\big)^J, A_1\bigoplus A_0\big)\to \extFunc^{n+1}\big(C^J, A_1\bigoplus A_0\big)
    \] 
    for each $n>0$. We note that:
    \begin{itemize}
        \item $\extFunc\big(\big(A_0\bigoplus A_0\big)^J, A_1\bigoplus A_0\big)$ vanishes by the previous paragraph;
        \item $\extFunc^n\big(\big( A_0\bigoplus  A_0\big)^J,  A_1\big)$ and $\extFunc^n(C^J, A_1) $ vanish for $n>1$ by the dual of Lemma~\ref{lmm: good};
        \item $\extFunc^n\big(\big( A_0\bigoplus  A_0\big)^J,  A_0\big)$ and $\extFunc^n(C^J, A_0)$ vanish for $n>1$ by hypothesis.
    \end{itemize}
    Since the bifunctor $\extFunc^n$ is additive for each $n>0$, $(A_1\bigoplus A_0)\in\text{Prod}(A_1\bigoplus A_0)^{\perp_{>0}}$. Moreover, it follows from Condition~\ref{cdt: mono} that $\text{Prod}(A_1\bigoplus A_0)\subseteq \text{Prod}(A_1\bigoplus A_0)^{\perp_{>0}}$. 
\end{proof}

\begin{Dfn}\cite[Definition~3.2]{hügel2025mutation}\label{dfn: mutation}
    Let $C,C'\in\extriangulatedCat$ be producers of silting subcategories of $\extriangulatedCat$, and let $\mathcal{D}=\text{Prod}(C)\cap\text{Prod}(C')$. We say that
    \begin{enumerate}[label=\DfnLbl]
        \item $C'$ is a \emph{left mutation} of $C$ \emph{(with respect to $\mathcal{D}$)} if there exists an $\additiveReal$-conflation
            \[
                C\xrightarrow[]{f} B_0\to B_1
            \] 
            such that $f$ is a left $\mathcal{D}$-approximation of $C$ and $\text{Prod}(B_0\bigoplus B_1)=\text{Prod}(C')$.

        \item $C'$ is a \emph{right mutation} of $C$ \emph{(with respect to $\mathcal{D}$)} if there exists an $\additiveReal$-conflation
            \[
                A_1\to A_0\xrightarrow[]{g} C
            \] 
            such that $g$ is a right $\mathcal{D}$-approximation of $C$ and $\text{Prod}(A_1\bigoplus A_0)=\text{Prod}(C')$.
    \end{enumerate}
\end{Dfn}

\begin{Prp}[]\label{prp: producer mutation}
    Suppose that Condition~\ref{cdt: mono} is satisfied. Let $C,C'\in\extriangulatedCat$ be producers of silting subcategories of $\extriangulatedCat$, and let $\mathcal{D}=\text{Prod}(C)\cap\text{Prod}(C')$. Then the following conditions are equivalent.
    \begin{enumerate}[label=\PrpLbl]
        \item $C'$ is a left mutation of $C$. 
        \item $C$ is a right mutation of $C'$. 
        \item $\mathcal{D}$ is a good contravariantly finite subcategory of $\text{Prod}(C')$ and $\text{Prod}(C)$ is the right product-mutation of $\text{Prod}(C')$ with respect to $\mathcal{D}$. 
        \item $\mathcal{D}$ is a good covariantly finite subcategory of $\text{Prod}(C)$ and $\text{Prod}(C')$ is the left product-mutation of $\text{Prod}(C)$ with respect to $\mathcal{D}$. 
    \end{enumerate}
\end{Prp}
\begin{proof}
    The equivalence (a)$\Leftrightarrow$(b) follows by proceeding as in the proof of \cite[Corollary~3.7]{hügel2025mutation}, whereas (c)$\Leftrightarrow$(d) follows directly from Theorem~\ref{thm: mutation}(d) and its analogue for left product-mutation. It thus suffices to prove (b)$\Leftrightarrow$(c). 

    Suppose that (c) holds. Then there exists an $\additiveReal$-conflation $A_1\to A_0\xrightarrow[]{g}C'$ with $g$ a right $\mathcal{D}$-approximation of $C'$. By Lemma~\ref{lmm: mutation}, $\text{Prod}(A_1\bigoplus A_0)$ is a silting subcategory of $\extriangulatedCat$. Moreover, it is contained in $\prod\nolimits^R(\text{Prod}(C');\mathcal{D})$ by definition, which is equal to $\text{Prod}(C)$ by assumption. It thus follows from the maximality property of silting subcategories that $\text{Prod}(A_1\bigoplus A_0)=\text{Prod}(C)$, which proves (b).

    On the other hand, suppose there exists an $\additiveReal$-conflation $A_1\to A_0\xrightarrow[]{g}C'$ such that $g$ is a right $\mathcal{D}$-approximation of $C'$ and $\text{Prod}(A_1\bigoplus A_0)=\text{Prod}(C)$. We show that $\mathcal{D}$ is a good contravariantly finite subcategory of $\text{Prod}(C')$. Let $S\in\text{Prod}(C')$. Then there exist a set $J$ and an $\additiveReal$-conflation $M\to (C')^J\to S$ whose morphisms on the left and right are the canonical inclusion and projection, respectively, by which $M\in\text{Prod}(C')$. Since $\extriangulatedTriple$ has exact products, we may consider the $\additiveReal$-conflation $A_1^J\to A_0^J\xrightarrow[]{g^J} (C')^J$. By the dual of the extriangulated octahedral axiom, we have a commutative diagram of $\additiveReal$-conflations
    \begin{center}
        \begin{tikzcd}
            A_1^J \arrow[equals]{d}[]{} \arrow[]{r}[]{} &E \arrow[]{d}[]{} \arrow[]{r}[]{} &M \arrow[]{d}[]{} \\
            A_1^J \arrow[]{r}[]{} &A_0^J \arrow[]{d}[]{g'} \arrow[]{r}[]{g^J} &(C')^J \arrow[]{d}[]{} \\
                                       &S \arrow[equals]{r}[]{} &S.
        \end{tikzcd}
    \end{center}
    Since $\text{Prod}(C')$ is silting, we know that $M,S\in\text{Prod}(C')^{\perp_{>0}}\subseteq \mathcal{D}^{\perp_{>0}}$. Note that $\text{Prod}(\mathcal{D})=\mathcal{D}$ implies that $g^J$ is a right $\mathcal{D}$-approximation of $(C')^J$. It then follows from the dual of Lemma~\ref{lmm: good} that $A_1^J\in\mathcal{D}^{\perp_{>0}}$, which is closed under extensions in $\extriangulatedCat$, by which $E\in\mathcal{D}^{\perp_{>0}}$. In particular, by applying $\text{Hom}_{\extriangulatedCat}(\mathcal{D},-)$ to the middle column in the previous diagram we have that $g'$ is a right $\mathcal{D}$-approximation of $S$. 

    By our assumption and Lemma~\ref{lmm: mutation}, we know that $\text{Prod}(A_1\bigoplus A_0)$ is a silting subcategory of $\extriangulatedCat$, which is contained in $\prod\nolimits^R(\text{Prod}(C');\mathcal{D})$ by definition. Since $\text{Prod}(\mathcal{D})=\mathcal{D}$, we have by Theorem~\ref{thm: mutation}(a) that $\prod\nolimits^R(\text{Prod}(C');\mathcal{D})$ is a presilting subcategory of $\extriangulatedCat$. Thus, it follows from the maximality property of silting subcategories that $\prod\nolimits^R(\text{Prod}(C');\mathcal{D})=\text{Prod}(A_1\bigoplus A_0)$, which is equal to $\text{Prod}(C)$ by assumption. This proves (b)$\Rightarrow$(c).
\end{proof}

\begin{Rmk}\label{rmk: producers}
    Note that the existence of an $\extFunc$-injective cogenerator of $\extriangulatedCat$ and the assumption that all objects in $\extriangulatedCat$ have finite $\extFunc$-injective dimension have not been used in this short section. However, when adding these hypotheses, we recover the bijection between product-closed silting subcategories and their producers from Proposition~\ref{prp: bijections}. In this case, Proposition~\ref{prp: producer mutation} establishes that the mutation of product-closed silting subcategories induces a mutation of producers of such subcategories, and moreover that these operations are equivalent.
\end{Rmk}

\subsection{Some special cases}\label{ssec: special}

We collect here some results which give a more detailed description of Proposition~\ref{prp: bijections} in some special cases which appear in Section~\ref{sec: examples}. 

The following result will be useful in the upcoming sections.

\begin{Lmm}[Dimension shifting]\label{lmm: shifting}
    Let $\extriangulatedTriple$ have positive extensions and $\subcat\subseteq \extriangulatedCat\ni C$.
    \begin{enumerate}[label=\LmmLbl]
    
        \item Suppose that $\{L_{i+1}\to S_i\to L_i\}_{i\ge 0}$ is an $\subcat$-resolution of $C$. Then, for all $D\in\subcat^{\perp_{>0}}, n\ge 0$ and $k>0$, we have that
            \[
            \extFunc^{n+k}(C,D)\simeq \extFunc^k(L_n,D).
            \] 
            
        \item Suppose that $\{L_i\to S_i\to L_{i+1}\}_{i\ge 0}$ is an $\subcat$-coresolution of $C$. Then, for all $B\in{}^{\perp_{>0}}\subcat, n\ge 0$ and $k>0$, we have that
            \[
            \extFunc^{n+k}(B,C)\simeq \extFunc^k(B,L_n).
            \] 
    \end{enumerate}
\end{Lmm}
\begin{proof}
    We only prove (a); the proof of (b) is analogous. Let $D\in\subcat^{\perp_{>0}}$. We will proceed by induction on $n$. The case $n=0$ is trivial. Assume the isomorphisms hold for some $n\ge0$ and let $k>0$. By the induction hypothesis, we know that $\extFunc^{n+(k+1)}(C,D)\simeq \extFunc^{k+1}(L_n,D)$. Moreover, by applying $\text{Hom}_\extriangulatedCat(-,D)$ to the $\additiveReal$-conflation $L_{n+1}\to S_n\to L_n$ we obtain the exact sequence
    \[
    \extFunc^k(S_n,D)\to \extFunc^k(L_{n+1},D)\to \extFunc^{k+1}(L_n,D)\to \extFunc^{k+1}(S_n,D).
    \] 
    Since $S_n\in\subcat$, it follows that $\extFunc^{k+1}(L_n,D)\simeq \extFunc^k(L_{n+1},D)$, by which $\extFunc^{(n+1)+k}(C,D)\simeq \extFunc^k(L_{n+1},D)$.
\end{proof}

\subsubsection{Enough $\extFunc$-projectives}\label{sssec: enough}

\begin{Dfn}\label{dfn: syzygies}\cite[\S 5.1]{liu2019hearts}
    Let $\subcat\subseteq \extriangulatedCat$. If $\extriangulatedCat$ has enough $\extFunc$-projectives, we recursively define the classes of \emph{$i$-th syzygies} of $\subcat$ as
    \[
    \resizebox{\hsize}{!}{$\syz[i]{\subcat} := \begin{cases} \subcat &\text{if } i=0, \\ \{C\in\extriangulatedCat \mid \exists \text{ an } \additiveReal\text{-conflation } C\to P\to S' \text{ with } P\in\text{Proj}_\extFunc(\extriangulatedCat) \text{ and } S'\in\syz[i-1]{\subcat}\} &\text{if } i>0. \end{cases}$}
    \]
    If $\extriangulatedCat$ has enough $\extFunc$-injectives, the classes $\cosyz[i]{\subcat}$ of \emph{$i$-th cosyzygies} of $\subcat$ are defined dually.
\end{Dfn}

We can now give a couple of examples of extriangulated categories $\extriangulatedTriple$ with positive extensions for which we can explicitly calculate the biadditive funtors $\extFunc^n$ for $n>1$.

\begin{Prp-Dfn}\label{prp-dfn: positive}
    \leavevmode
    \begin{enumerate}[label=\ExpLbl]
        \item If $\extriangulatedCat$ has enough $\extFunc$-projectives, we recursively define the additive bifunctors $\extFunc^n_\Omega:\extriangulatedCat^\text{op}\times\extriangulatedCat\to \text{Ab}$ for each $n\ge0$ and $A,C\in\extriangulatedCat$ as
            \[
                \extFunc^n_\Omega(C,A) := \begin{cases} \text{Hom}_\extriangulatedCat(C,A) &\text{if } n=0, \\ \extFunc(\syz[n-1]{C},A) &\text{if } n>0, \end{cases}
            \] 
            which do not depend on a particular choice of syzygy by \cite[Proposition~3.4]{herschend2022n-exangulated}. Then, $\extriangulatedTriple$ has positive extensions by \cite[Definition~3.12~\&~Proposition~3.20]{herschend2022n-exangulated}. 

        \item If $\extriangulatedCat$ has enough $\extFunc$-injectives, we recursively define additive bifunctors $\extFunc^n_\Sigma:\extriangulatedCat^\text{op}\times\extriangulatedCat\to \text{Ab}$ for each $n\ge0$ and $A,C\in\extriangulatedCat$ as
            \[
                \extFunc^n_\Sigma(C,A) := \begin{cases} \text{Hom}_\extriangulatedCat(C,A) &\text{if } n=0, \\ \extFunc(C,\cosyz[n-1]{A}) &\text{if } n>0. \end{cases}
            \] 
            By the same results cited in (1), these do not depend on a particular choice of cosyzygy, and they imply that $\extriangulatedTriple$ has positive extensions. In particular, if $\extriangulatedCat$ has both enough $\extFunc$-projectives and $\extFunc$-injectives, it follows from \cite[Lemma~5.1]{liu2019hearts} that $\extFunc^n_\Omega\simeq\extFunc^n_\Sigma$ for each $n\ge0$. 
    \end{enumerate}
\end{Prp-Dfn}

The following result is obtained using standard arguments with syzygies.

\begin{Lmm}\label{lmm: generated}
    Let $\extriangulatedCat$ have enough $\extFunc$-projectives and $\subcat\subseteq\extriangulatedCat$ be such that it is closed under cocones in $\extriangulatedCat$ and $\text{Proj}_\extFunc(\extriangulatedCat)\subseteq\subcat$. Then $\subcat^{\perp_1}=\subcat^{{\perp_{>0}}}$.
\end{Lmm}

\begin{Lmm}[]\label{lmm: generated'}
    Let $\extriangulatedCat$ have both enough $\extFunc$-projectives and $\extFunc$-injectives, and $C\in\extriangulatedCat$. Then the cotorsion pair in $\extriangulatedCat$ generated by ${}^{\perp_{>0}}C$ is hereditary and equal to $({}^{\perp_{>0}}C,({}^{\perp_{>0}}C)^{\perp_1})$.
\end{Lmm}
\begin{proof}
    Recall that $({}^{\perp_1}(({}^{\perp_{>0}}C)^{\perp_1}),({}^{\perp_{>0}}C)^{\perp_1})$ is the cotorsion pair in $\extriangulatedCat$ generated by ${}^{\perp_{>0}}C$. Note that ${}^{\perp_{>0}}C$ is closed under cocones in $\extriangulatedCat$ by the exact sequences in Definition~\ref{dfn: positive}, and moreover that $\text{Proj}_\extFunc(\extriangulatedCat)={}^{\perp_{>0}}\extriangulatedCat\subseteq{}^{\perp_{>0}}C$ by Proposition-Definition~\ref{prp-dfn: positive}. Lemma~\ref{lmm: generated} then gives us that $({}^{\perp_{>0}}C)^{\perp_1}=({}^{\perp_{>0}}C)^{\perp_{>0}}$, which is closed under cones in $\extriangulatedCat$. By its dual, we conclude that
    \[
        {}^{\perp_1}(({}^{\perp_{>0}}C)^{\perp_1})={}^{\perp_{>0}}(({}^{\perp_{>0}}C)^{\perp_1})={}^{\perp_{>0}}(({}^{\perp_{>0}}C){}^{\perp_{>0}})={}^{\perp_{>0}}C.
    \] 
\end{proof}

\begin{Prp}[]\label{prp: generated}
    Let $\extriangulatedTriple$ have exact products, an $\extFunc$-injective cogenerator $Q$, and all objects with finite $\extFunc$-injective dimension, and let $C\in\extriangulatedCat$ be a producer of a silting subcategory of $\extriangulatedCat$. If $\extriangulatedCat$ also has enough $\extFunc$-projectives, then $(\text{Prod}(C)^\lor,\text{Prod}(C)^\land)=({}^{\perp_{>0}}C,({}^{\perp_{>0}}C)^{\perp_1})$, which equals the cotorsion pair in $\extriangulatedCat$ generated by ${}^{\perp_{>0}}C$.
\end{Prp}
\begin{proof}
    Recall that $(\text{Prod}(C)^\lor,\text{Prod}(C)^\land)$ is a cotorsion pair in $\extriangulatedCat$ by Theorem~\ref{thm: AT}. Suppose that $\extriangulatedCat$ also has enough $\extFunc$-projectives. By Lemma~\ref{lmm: generated'} and the fact that cotorsion pairs are uniquely determined by either of their classes, it suffices to show that $({}^{\perp_{>0}}C)^{\perp_{>0}}=\text{Prod}(C)^\land$.

    Since $\text{Prod}(C)$ is a presilting subcategory of $\extriangulatedCat$, by applying \cite[Proposition~4.7(3)]{adachi2022hereditary} to $\mathcal{W}=\text{Prod}(C)=\mathcal{X}$ we have that $\text{Prod}(C)^\land$ is the smallest subcategory of $\extriangulatedCat$ which contains $\text{Prod}(C)$ and is closed under summands, extensions and cones in $\extriangulatedCat$, and dually for $\text{Prod}(C)^\lor$. Moreover, it follows from Remark~\ref{rmk: heredity}(1) that $\text{Prod}(C)^\lor\subseteq{}^{\perp_{>0}}C$. Furthermore, since $\extriangulatedTriple$ has products and $C\in({}^{\perp_{>0}}C)^{\perp_{>0}}$ trivially, Remark~\ref{rmk: products}(1) implies that $\text{Prod}(C)\subseteq({}^{\perp_{>0}}C)^{\perp_{>0}}$, by which $\text{Prod}(C)^\land\subseteq({}^{\perp_{>0}}C)^{\perp_{>0}}$. The first inclusion then gives us
    \begin{align*}
        ({}^{\perp_{>0}}C)^{\perp_{>0}} &\subseteq (\text{Prod}(C)^\lor)^{\perp_{>0}} \\
                                        &\subseteq (\text{Prod}(C)^\lor)^{\perp_1} \\
                                        &= \text{Prod}(C)^\land,
    \end{align*}
    which completes the proof.
\end{proof}

\subsubsection{Finite positive global dimension}\label{ssec: dimension}

\begin{Dfn}\cite[Definition~3.28]{gorsky2021positive}\label{dfn: global}
    Let $\extriangulatedTriple$ have positive extensions. If there exists $m\in\mathbb{Z}_{\ge0}$ such that $\extFunc^m\neq 0$ and $\extFunc^{m+1}=0$ (equivalently, $\extFunc^m\neq 0$ and $\extFunc^n=0$ for all $n>m$), the \emph{positive global dimension} of $\extriangulatedTriple$ is finite and equals $m$; otherwise, it is $\infty$.
\end{Dfn}

\begin{Lmm}\label{lmm: dimension}
    Let $\extriangulatedTriple$ have positive global dimension $m\in\mathbb{Z}_{\ge0}$. Suppose that $\subcat$ is a presilting subcategory of $\extriangulatedCat$. Then $\subcat^\land=\subcat^\land_m$ and $\subcat^\lor=\subcat^\lor_m$.
\end{Lmm}
\begin{proof}
    We prove the equality $\subcat^\land=\subcat^\land_m$; the case for $\subcat^\lor$ is analogous. By definition, it suffices to show that $\subcat^\land\subseteq \subcat^\land_m$, which we will do by induction. Since $\subcat$ is closed under summands in $\extriangulatedCat$, $0\in\subcat$. Thus, $\subcat^\land_0\subseteq \subcat^\land_m$. Suppose that $\subcat^\land_k\subseteq \subcat^\land_m$ for some $k\ge 0$ and let $C\in\subcat^\land_{k+1}$. Then there exists a sequence of $\additiveReal$-conflations 
    \begin{align*}
        L_1\to  &S_0\to L_0, \\
        L_2\to  &S_1\to L_1, \\
                &\vdots \\
        L_{k+1}\to  &S_k\to L_k, \\
        0\to  &S_{k+1}\to L_{k+1},
    \end{align*}
    with $L_0=C$ and $S_i\in\subcat$ for each $0\le i\le k+1$. In particular, this implies that $L_1\in\subcat^\land_k$. By the induction hypothesis, $L_1$ then has an $\subcat$-resolution of length $m$, which we can append to the first $\additiveReal$-conflation in the previous sequence in order to produce an $\subcat$-resolution of $C$ of length $m+1$. Let us relabel this $\subcat$-resolution of $C$ as
    \begin{align*}
        L_1'\to  &S_0'\to C, \\
        L_2'\to  &S_1'\to L_1', \\
                &\vdots \\
        L_{m+1}'\to  &S_m'\to L_m', \\
        0\to  &S_{m+1}'\to L_{m+1}'.
    \end{align*}
    From the last $\additiveReal$-conflation, it follows that $L_{m+1}'\simeq S_{m+1}\in\subcat$. Moreover, by \hyperref[lmm: shifting]{dimension shifting} and the assumption that $\extriangulatedTriple$ has finite positive global dimension $m$, we have that $\extFunc(L_m',L_{m+1}')\simeq \extFunc^{m+1}(C,L_{m+1}')=0$, by which $S_m'\simeq L_{m+1}'\bigoplus L_m'$. Since $\subcat$ is closed under summands in $\extriangulatedCat$, we have that $L_m'\in\subcat$, and by replacing the last two $\additiveReal$-conflations in the previous sequence with $0\to L_m'\to L_m'$ we obtain that $C\in\subcat^\land_m$. Thus, $\subcat^\land_{k+1}\subseteq \subcat^\land_m$. By induction on $k$, we conclude that $\subcat^\land=\subcat^\land_m$.
\end{proof}

\begin{Lmm}[]\label{lmm: n-(co)tilting}\leavevmode
    Let $\extriangulatedCat$ have enough $\extFunc$-projectives and positive global dimension $m\in\mathbb{Z}_{\ge0}$. Then $\text{Proj}_\extFunc(\extriangulatedCat)^\land_m=\extriangulatedCat$.
\end{Lmm}
\begin{proof}
    Let $C\in\extriangulatedCat$. Note that, since $\extriangulatedCat$ has enough $\extFunc$-projectives, we may consider a $\text{Proj}_\extFunc(\extriangulatedCat)$-resolution of $C$, which in particular is of the form $\{\syz[i+1]{C}\to P_i\to \syz[i]{C}\}_{i\in\mathbb{Z}_{\ge0}}$. Since $\extriangulatedTriple$ has positive global dimension $m$, it follows from Proposition-Definition \ref{prp-dfn: positive}(1) that
    \[
        0 = \extFunc^{m+1}(C,C') = \extFunc(\syz[m]{C},C') 
    \] 
    for all $C'\in\extriangulatedCat$, by which $\syz[m]{C}\in\text{Proj}_\extFunc(\extriangulatedCat)$. Thus, by replacing the $m$-th $\additiveReal$-conflation in the previous sequence with $0\to \syz[m]{C}\to \syz[m]{C}$, we obtain that $C\in\text{Proj}_\extFunc(\extriangulatedCat)^\land_m$.
\end{proof}

\begin{Prp}\label{prp: bounded}
    Let $\extriangulatedCat$ have enough $\extFunc$-projectives, enough $\extFunc$-injectives and finite positive global dimension. Then all cotorsion pairs in $\extriangulatedCat$ are bounded.
\end{Prp}
\begin{proof}
    Recall that, for any cotorsion pair $\cotorsionPair$ in $\extriangulatedCat$, we have that $\text{Proj}_\extFunc(\extriangulatedCat)\subseteq\cotorsionClass$ and $\text{Inj}_{\extFunc}(\extriangulatedCat)\subseteq\cotorsionfreeClass$ by Remark~\ref{rmk: orthogonality}. Since Lemma~\ref{lmm: n-(co)tilting} and its dual imply that $\text{Proj}_\extFunc(\extriangulatedCat)^\land=\extriangulatedCat=\text{Inj}_\extFunc(\extriangulatedCat)^\lor$, the assertion follows.
\end{proof}

\begin{Prp}[]\label{prp: n-(co)tilting}
    Let $\extriangulatedCat$ have enough $\extFunc$-projectives and positive global dimension $m\in\mathbb{Z}_{\ge0}$. Then a subcategory $\subcat$ of $\extriangulatedCat$ is silting if, and only if, it is $m$-tilting.
\end{Prp}
\begin{proof}
    Let $\subcat$ be a silting subcategory of $\extriangulatedCat$. Then $(\subcat^\lor,\subcat^\land)$ is a cotorsion pair in $\extriangulatedCat$ by Theorem~\ref{thm: AT}, and in particular $\text{Proj}_\extFunc(\extriangulatedCat)\subseteq\subcat^\lor$. By Lemma~\ref{lmm: dimension}, it follows that $\text{Proj}_\extFunc(\extriangulatedCat)\subseteq\subcat^\lor_m$, whereas $\subcat\subseteq\text{Proj}_\extFunc(\extriangulatedCat)^\land_m$ by Lemma~\ref{lmm: n-(co)tilting}.

    Now, let $\subcat$ be an $m$-tilting subcategory of $\extriangulatedCat$. Then, in particular, $\subcat$ is a presilting subcategory of $\extriangulatedCat$ and $\text{Proj}_\extFunc(\extriangulatedCat)\subseteq\subcat^\lor_m$. Since $\text{Proj}_\extFunc(\extriangulatedCat)^\land_m\subseteq(\subcat^\lor_m)^\land_m\subseteq\text{thick}(\subcat)$, we have that $\text{thick}(\subcat)=\extriangulatedCat$ by Lemma~\ref{lmm: n-(co)tilting}.
\end{proof}

\begin{Rmk}\label{rmk: n-(co)tilting}
    It follows from \cite[Proposition~2.1~\&~Theorem~4.3]{gorsky2023hereditary} that, if $\extriangulatedCat$ has enough $\extFunc$-projectives, enough $\extFunc$-injectives, and positive global dimension $m\in\{0,1\}$, then silting subcategories, $m$-tilting subcategories, and $m$-cotilting subcategories are all equivalent. Proposition~\ref{prp: n-(co)tilting} and its dual then extend this result for all $m\in\mathbb{Z}_{\ge0}$.
\end{Rmk}

\section{Examples}\label{sec: examples}

Throughout this section, let 
\begin{itemize}
    \item $\AUR$ be an associative unital ring;
    \item $Q$ be an injective cogenerator of $\text{Mod}(\AUR)$ (e.g., $Q=\text{Hom}_\mathbb{Z}(\AUR,\mathbb{Q}/\mathbb{Z})$);
    \item $\derivedCat(\AUR)$ be the unbounded derived category of $\text{Mod}(\AUR)$;
    \item $\homotopyCat^b(\text{Inj}(\AUR))$ be the category of cochain complexes of injective $\AUR$-modules that are non-zero in finitely many degrees, considered up to homotopy;
    \item $\Kzn{\AUR}$ be the subcategory of complexes in $\homotopyCat^b(\text{Inj}(\AUR))$ concentrated in the interval $[0,n]$ (i.e., those which are zero elsewhere) for a fixed $n\in\mathbb{Z}_{>0}$.
\end{itemize}

\subsection{A canonical construction}\label{ssec: construction}

We now describe a general procedure for constructing, from a triangulated category with products containing a presilting subcategory with a producer, an extriangulated category which satisfies all of the hypotheses considered throughout Section~\ref{sec: large}. Throughout this section, let $\triangulatedTriple$ be a triangulated category. 

\begin{Rmk}\label{rmk: triangulated}
    By \cite[Proposition~3.22(1)]{nakaoka2019extriangulated}, $\triangulatedCat$ is an extriangulated category with additive bifunctor $\text{Ext}_{\triangulatedCat}^{1}(\ast,-)=\text{Hom}_\triangulatedCat(\ast,\shiftFunc{-})$ and conflations given by sequences of morphisms of the form $A\to B\to C$ which extend to a distinguished triangle $A\to B\to C\to \shiftFunc[1]{A}$. Moreover, it has positive extensions given by $\text{Ext}_{\triangulatedCat}^{n}(\ast,-)=\text{Hom}_\triangulatedCat(\ast,\shiftFunc[n]{-})$ for each $n>1$ by \cite[Corollary~3.23]{gorsky2021positive}. Furthermore, if it has products, they are exact by \cite[Proposition~1.2.1]{neeman2014triangulated}. In the following, we will consider all triangulated categories with this extriangulated structure, which we denote by $(\triangulatedCat,\text{Ext}_{\triangulatedCat}^{1},\Delta)$.
\end{Rmk}

\begin{Dfn}\label{dfn: interval}
    Let $\subcat$ be a presilting subcategory of $\triangulatedCat$. For any integers $p\le q$, we denote
    \[
        \subcat^{[p,q]}:=\shiftFunc[p]{S}\ast\shiftFunc[p+1]{\subcat}\ast\dots\ast\shiftFunc[q]{\subcat}.
    \] 
\end{Dfn}

\begin{Lmm}[]\label{lmm: construction}
    Let $\subcat\subseteq\triangulatedCat$ be such that $\subcat\subseteq{}^{\perp_{>0}}\subcat$. The following conditions are satisfied for $m\in\mathbb{Z}_{>0}$.
    \begin{enumerate}[label=\LmmLbl]
        \item $\subcat^{[-m,0]}\subseteq{}^{\perp_{>0}}\subcat$ and $\subcat^{[-m,0]}\subseteq(\shiftFunc[-m]{\subcat})^{\perp_{>0}}$.
    
        \item $\subcat^{[-m,0]}$ is closed under extensions in $\triangulatedCat$. 
    \end{enumerate}
\end{Lmm}
\begin{proof}\leavevmode
    \begin{enumerate}[label=\LmmLbl]
    \item Let $S\in\subcat$. We will proceed by induction on $m$. If $C\in\subcat^{[-1,0]}$, then there exists a distinguished triangle $B\to C\to D\to \shiftFunc[1]{B}$ with $D\in\subcat$ and $B\in\shiftFunc[-1]{\subcat}$, to which we can apply $\text{Hom}_\triangulatedCat(-,S)$ in order to obtain an exact sequence of the form
            \begin{equation}\label{eq: sequence}
                \text{Hom}_\triangulatedCat(\shiftFunc[1]{B},\shiftFunc[n]{S})\to \text{Hom}_\triangulatedCat(D,\shiftFunc[n]{S})\to \text{Hom}_\triangulatedCat(C,\shiftFunc[n]{S})\to \text{Hom}_\triangulatedCat(B,\shiftFunc[n]{S})
            \end{equation}
            for each $n>0$. Since $\subcat\subseteq{}^{\perp_{>0}}\subcat$ by hypothesis, $D,\shiftFunc[1]{B}\in\subcat$, and $\text{Hom}_\triangulatedCat(B,\shiftFunc[n]{S})\simeq\text{Hom}_\triangulatedCat(\shiftFunc[1]{B},\shiftFunc[n+1]{S})$ for each $n>0$, it follows that $C\in{}^{\perp_{>0}}\subcat$. Thus, $\subcat^{[-1,0]}\subseteq{}^{\perp_{>0}}\subcat$.

            Now, suppose that $\subcat^{[-k,0]}\subseteq{}^{\perp_{>0}}\subcat$ for some $k\ge1$ and let $C\in\subcat^{[-(k+1),0]}$. By associativity of $\ast$, there exists a distinguished triangle $B\to C\to D\to \shiftFunc[1]{B}$ with $D\in\subcat$ and $B\in(\shiftFunc[-(k+1)]{\subcat}\ast\shiftFunc[-k]{\subcat}\ast\dots\ast\shiftFunc[-1]{\subcat})=\shiftFunc[-1]{\subcat^{[-k,0]}}$. By applying $\text{Hom}_\triangulatedCat(-,S)$ to it, we obtain an exact sequence of the form (\ref{eq: sequence}) for each $n>0$. Since $D\in\subcat$ and $\shiftFunc[1]{B}\in\subcat^{[-k,0]}$, it follows from the induction hypothesis that $\subcat^{[-(k+1),0]}\subseteq{}^{\perp_{>0}}\subcat$. The inclusion $\subcat^{[-m,0]}\subseteq(\shiftFunc[-m]{\subcat})^{\perp_{>0}}$ is proved analogously.

        \item We first prove that $\subcat\ast\subcat\subseteq\subcat$. Let $S\to E\to S'\to \shiftFunc[1]{S}$ be a distinguished triangle with $S,S'\in\subcat$. Then $\text{Ext}_{\triangulatedCat}^{1}(S',S)=0$, which implies that $E=S\bigoplus S'$ in $\triangulatedCat$. By our assumption that all subcategories of an additive category are additive subcategories, it follows that $E\in\subcat$. 

        Now, let $m\in\mathbb{Z}_{>0}$. By the proof of \cite[Lemma~2.15(b)]{aihara2012silting}, we know that $\shiftFunc[q]{\subcat}\ast\shiftFunc[p]{\subcat}\subseteq\shiftFunc[p]{\subcat}\ast\shiftFunc[q]{\subcat}$ whenever $p\le q$. Since $\subcat\ast\subcat\subseteq\subcat$ implies that $\shiftFunc[z]{\subcat}\ast\shiftFunc[z]{\subcat}\subseteq\shiftFunc[z]{\subcat}$ for each $z\in\mathbb{Z}$, and moreover $\ast$ is associative, it follows that
            \[
                \subcat^{[-m,0]}\ast\subcat^{[-m,0]} \subseteq (\shiftFunc[-m]{\subcat}\ast\shiftFunc[-m]{\subcat})\ast(\shiftFunc[-(m-1)]{\subcat}\ast\shiftFunc[-(m-1)]{\subcat})\dots\ast(\subcat\ast\subcat) \subseteq \subcat^{[-m,0]}.
            \]
    \end{enumerate}
\end{proof}

\begin{Prp}[]\label{prp: construction}
    Let $\subcat\subseteq\triangulatedCat$ be such that $\subcat\subseteq{}^{\perp_{>0}}\subcat$. The following conditions are satisfied for $m\in\mathbb{Z}_{>0}$.
    \begin{enumerate}[label=\PrpLbl]
        \item $(\subcat^{[-m,0]},\extFunc_{[-m,0]},\additiveReal_{[-m,0]})$ is an extriangulated category, where
            \begin{itemize}
                \item $\extFunc_{[-m,0]}(C,A):=\text{Ext}_{\triangulatedCat}^{1}(C,A)$ for all $A,C\in\subcat^{[-m,0]}$;
                \item $A\to B\to C$ is an $\additiveReal_{[-m,0]}$-conflation whenever $A\to B\to C\to \shiftFunc[1]{A}$ is a distinguished triangle in $\triangulatedCat$ with $A,B,C\in\subcat^{[-m,0]}$.
            \end{itemize}
            In particular, we have that $\subcat\subseteq\text{Inj}_{\extFunc_{[-m,0]}}(\subcat^{[-m,0]})$ and $\shiftFunc[-m]{\subcat}\subseteq\text{Proj}_{\extFunc_{[-m,0]}}(\subcat^{[-m,0]})$.

        \item In $(\subcat^{[-m,0]},\extFunc_{[-m,0]},\additiveReal_{[-m,0]})$, we have the equality
            \[
                \subcat^\lor_m=\subcat^{[-m,0]}=(\shiftFunc[-m]{\subcat})^\land_m.
            \] 
            In particular, it follows that this category has both enough $\extFunc_{[-m,0]}$-projectives and $\extFunc_{[-m,0]}$-injectives, which implies that it has positive extensions, given by
            \[
                (\extFunc_{[-m,0]})^n(C,A):=\extFunc_{[-m,0]}(\syz[n-1]{C},A)\simeq\extFunc_{[-m,0]}(C,\cosyz[n-1]{A})
            \] 
            for all $A,C\in\subcat^{[-m,0]}$ and $n>1$, as well as positive global dimension at most $m$.

        \item For all $A,C\in\subcat^{[-m,0]}$ and $n>1$, we have that
            \[
                (\extFunc_{[-m,0]})^n(C,A)\simeq\text{Ext}_{\triangulatedCat}^{n}(C,A).
            \] 
            
        \item Suppose that $\subcat$ is also closed under summands in $\triangulatedCat$, i.e., that $\subcat$ is a presilting subcategory of $\triangulatedCat$. Then $\text{Inj}_{\extFunc_{[-m,0]}}(\subcat^{[-m,0]})=\subcat$ and $\text{Proj}_{\extFunc_{[-m,0]}}(\subcat^{[-m,0]})=\shiftFunc[-m]{\subcat}$. Moreover, for $\subcat'\subseteq\subcat^{[-m,0]}$, $\subcat'$ is a presilting subcategory of $\subcat^{[-m,0]}$ if, and only if, it is a presilting subcategory of $\triangulatedCat$.

        \item Suppose that there exists $Q\in\triangulatedCat$ such that $\text{Prod}(Q)=\subcat$. Then $Q$ is an $\extFunc_{[-m,0]}$-injective cogenerator of $\subcat^{[-m,0]}$.
    \end{enumerate}
\end{Prp}
\begin{proof}\leavevmode
    \begin{enumerate}[label=\PrpLbl] 
        \item The first assertion follows from Lemma~\ref{lmm: construction}(b) and \cite[Remark~2.18]{nakaoka2019extriangulated}. The second one is then a consequence of Lemma~\ref{lmm: construction}(a).

        \item Let us first prove that $\subcat^\lor_m=\subcat^{[-m,0]}$ in $(\subcat^{[-m,0]},\extFunc_{[-m,0]},\additiveReal_{[-m,0]})$ by induction on $m$; the proof of the remaning equality is analogous. If $C\in\subcat^{[-1,0]}$, there exists a distinguished triangle $B\to C\to D\to \shiftFunc[1]{B}$ with $D\in\subcat$ and $B\in\shiftFunc[-1]{\subcat}$, which we can rotate into the distinguished triangle $C\to D\to \shiftFunc[1]{B}\to \shiftFunc[1]{C}$, inducing an $\additiveReal_{[-1,0]}$-conflation $C\to D\to \shiftFunc[1]{B}$ with $D,\shiftFunc[1]{B}\in\subcat$. Thus, $C\in\subcat^\lor_1$. Now, suppose that the assertion is true for some $k\ge1$ and let $C'\in\subcat^{[-(k+1),0]}$. Similarly, there exists a distinguished triangle $B'\to C'\to D'\to \shiftFunc[1]{B'}$ with $D'\in\subcat$ and $B'\in\shiftFunc[-1]{\subcat^{[-k,0]}}$, which rotates into $C'\to D'\to \shiftFunc[1]{B'}\to \shiftFunc[1]{C'}$, inducing an $\additiveReal_{[-(k+1),0]}$-conflation $C'\to D'\to \shiftFunc[1]{B'}$, where $\shiftFunc[1]{B'}\in\subcat^{[-k,0]}$. By the induction hypothesis, $\shiftFunc[1]{B'}$ has an $\subcat$-coresolution of length $k$ in $(\subcat^{[-k,0]},\extFunc_{[-k,0]},\additiveReal_{[-k,0]})$. Since $\subcat^{[-k,0]}\subseteq\subcat^{[-(k+1),0]}$, all $\additiveReal_{[-k,0]}$-conflations are $\additiveReal_{[-(k+1),0]}$-conflations by (a), so by appending $C'\to D'\to \shiftFunc[1]{B'}$ we obtain an $\subcat$-coresolution of $C'$ of length $k+1$ in $(\subcat^{[-(k+1),0]},\extFunc_{[-(k+1),0]},\additiveReal_{[-(k+1),0]})$. 

            Now, fix $m\in\mathbb{Z}_{>0}$. Since $\subcat\subseteq\text{Inj}_{\extFunc_{[-m,0]}}(\subcat^{[-m,0]})$ and $\shiftFunc[-m]{\subcat}\subseteq\text{Proj}_{\extFunc_{[-m,0]}}(\subcat^{[-m,0]})$ by (a), it follows that $\subcat^{[-m,0]}$ has both enough $\extFunc_{[-m,0]}$-projectives and $\extFunc_{[-m,0]}$-injectives. The positive extensions in the statement are then given by Proposition-Definition~\ref{prp-dfn: positive}. For the positive global dimension, consider $A,C\in\subcat^{[-m,0]}$. Note that there exist $\additiveReal_{[-m,0]}$-conflations
            \begin{align*}
                A\to &I_1\to \cosyz[1]{A}, \\
                \cosyz[1]{A}\to &I_2\to \cosyz[2]{A}, \\
                                   &\vdots \\
                \cosyz[m]{A}\to &I_m\to 0,
            \end{align*}
            with $I_i\in\text{Inj}_{\extFunc_{[-m,0]}}(\subcat^{[-m,0]})$ for $1\le i\le m$, which implies that $(\extFunc_{[-m,0]})^{m+1}(C,A)\simeq\extFunc_{[-m,0]}(C,\cosyz[m]{A})$ vanishes by the last $\additiveReal_{[-m,0]}$-conflation. Thus, $(\extFunc_{[-m,0]})^{m+1}=0$, by which $(\subcat^{[-m,0]},\extFunc_{[-m,0]},\additiveReal_{[-m,0]})$ has positive global dimension at most $m$.

        \item Let $m\in\mathbb{Z}_{>0}$ and $A,C\in\subcat^{[-m,0]}$.

            We proceed by induction. By (b), we may consider an $\additiveReal_{[-m,0]}$-conflation $A\to S_1\to \cosyz[1]{A}$ with $S_1\in\subcat$, as well as its corresponding distinguished triangle $A\to S_1\to \cosyz[1]{A}\to \shiftFunc[1]{A}$, which we can rotate into $S_1\to \cosyz[1]{A}\to \shiftFunc[1]{A}\to \shiftFunc[1]{S_1}$. By applying $\text{Hom}_{\triangulatedCat}(C,-)$ we obtain, for each $n>0$, an exact sequence of the form
            \[
                \text{Ext}_\triangulatedCat^n(C,S_1)\to \text{Ext}_\triangulatedCat^n(C,\cosyz[1]{A})\to \text{Ext}_\triangulatedCat^n(C,\shiftFunc[1]{A})\to \text{Ext}_\triangulatedCat^{n+1}(C,S_1),
            \] 
            whose leftmost and rightmost terms vanish by Lemma~\ref{lmm: construction}(a). Thus, we have that $\text{Ext}_{\triangulatedCat}^{n}(C,\cosyz[1]{A})\simeq\text{Ext}_{\triangulatedCat}^{n}(C,\shiftFunc[1]{A})$ for each $n>0$. By (b), it then follows that
            \begin{align*}
                (\extFunc_{[-m,0]})^2(C,A) &\simeq \extFunc_{[-m,0]}(C,\cosyz[1]{A}) \\
                                    &= \text{Ext}_{\triangulatedCat}^{1}(C,\cosyz[1]{A}) \\
                                    &\simeq \text{Ext}_{\triangulatedCat}^{1}(C,\shiftFunc[1]{A}) \\
                                    &\simeq \text{Ext}_{\triangulatedCat}^{2}(C,A).
            \end{align*}
            Now, suppose that $(\extFunc_{[-m,0]})^k(Z,X)\simeq\text{Ext}_{\triangulatedCat}^{k}(Z,X)$ for all $X,Z\in\subcat^{[-m,0]}$ and some $2\le k<m$. Then, since $\cosyz[k]{A}=\cosyz[1]{\cosyz[k-1]{A}}$, it follows from the same arguments that
            \begin{align*}
                (\extFunc_{[-m,0]})^{k+1}(C,A) &\simeq \extFunc_{[-m,0]}(C,\cosyz[k]{A}) \\
                                        &\simeq (\extFunc_{[-m,0]})^k(C,\cosyz[1]{A}) \\
                &\simeq \text{Ext}_{\triangulatedCat}^{k}(C,\cosyz[1]{A}) \\
                &\simeq \text{Ext}_{\triangulatedCat}^{k}(C,\shiftFunc[1]{A}) \\
                &\simeq \text{Ext}_{\triangulatedCat}^{k+1}(C,A).
            \end{align*}

        \item Let $m\in\mathbb{Z}_{>0}$ and $I\in\text{Inj}_{\extFunc_{[-m,0]}}(\subcat^{[-m,0]})$. By (b), there exists an $\additiveReal_{[-m,0]}$-conflation $I\to S_1\to L_1$ with $S\in\subcat$, which splits since $I$ is $\extFunc_{[-m,0]}$-injective, by which $S_1\simeq I\bigoplus L_1$ in $\subcat^{[-m,0]}$. By our assumption that $\subcat^{[-m,0]}$ is an additive subcategory of $\triangulatedCat$ and the hypothesis that $\subcat$ is closed under summands in $\triangulatedCat$, it follows that $I\in\subcat$. Thus, $\text{Inj}_{\extFunc_{[-m,0]}}(\subcat^{[-m,0]})\subseteq\subcat$, and furthermore $\text{Inj}_{\extFunc_{[-m,0]}}(\subcat^{[-m,0]})=\subcat$ by (a). The equality $\text{Proj}_{\extFunc_{[-m,0]}}(\subcat^{[-m,0]})=\shiftFunc[-m]{\subcat}$ is proved analogously.

            On the other hand, the last assertion follows from our hypothesis and (c).

        \item Let $C\in\subcat^{[-m,0]}$. By (b), there exists an $\additiveReal_{[-m,0]}$-conflation $C\to Q_0\to L_1$ with $Q_0\in\text{Prod}(Q)$. Hence, $Q^J\simeq Q_0\bigoplus Q_1$ in $\subcat^{[-m,0]}$ for some set $J$ and $Q_1\in\text{Prod}(Q)$, which gives the split $\additiveReal_{[-m,0]}$-conflation $Q_0\to Q^J\to Q_1$. By the extriangulated octahedral axiom, we can then compose the $\additiveReal_{[-m,0]}$-inflations $C\to Q_0$ and $Q_0\to Q^J$ in order to obtain an $\additiveReal_{[-m,0]}$-inflation $C\to Q^J$.
    \end{enumerate}
\end{proof}

As a consequence of these results, together with Proposition-Definition~\ref{prp-dfn: positive}(1) and Remark~\ref{rmk: products}(1), we obtain the following canonical construction.

\begin{Cor}\label{cor: construction}
    Let $\triangulatedTriple$ be a triangulated category with products, $Q\in\triangulatedCat$ be such that $\text{Prod}(Q)$ is a presilting subcategory of $\triangulatedCat$, and $m\in\mathbb{Z}_{>0}$. Then
    \[
        \big(\text{Prod}(Q)^{[-m,0]},\extFunc_{[-m,0]},\additiveReal_{[-m,0]}\big)
    \] 
    is an extriangulated category with
    \begin{itemize}
        \item $\extFunc_{[-m,0]}(C,A):=\text{Ext}_{\triangulatedCat}^{1}(C,A)$ for all $A,C\in\text{Prod}(Q)^{[-m,0]}$;
        \item $A\to B\to C$ an $\additiveReal_{[-m,0]}$-conflation whenever $A\to B\to C\to \shiftFunc[1]{A}$ is a distinguished triangle in $\triangulatedCat$ with $A,B,C\in\text{Prod}(Q)^{[-m,0]}$;
        \item $\text{Inj}_{\extFunc_{[-m,0]}}\big(\text{Prod}(Q)^{[-m,0]}\big)=\text{Prod}(Q)$ and $\text{Proj}_{\extFunc_{[-m,0]}}\big(\text{Prod}(Q)^{[-m,0]}\big)=\shiftFunc[-m]{\text{Prod}(Q)}$;
        \item exact products and an $\extFunc_{[-m,0]}$-injective cogenerator $Q$ \textemdash \ in particular, enough $\extFunc_{[-m,0]}$-injectives and positive extensions
            \[
                \big(\extFunc_{[-m,0]}\big)^n(C,A)\simeq\text{Ext}_{\triangulatedCat}^{n}(C,A)
            \] 
            for all $A,C\in\text{Prod}(Q)^{[-m,0]}$ and $n>1$;
        \item finite positive global dimension at most $m$ \textemdash \ hence, all objects of finite $\extFunc_{[-m,0]}$-injective dimension;
        \item enough $\extFunc_{[-m,0]}$-projectives \textemdash \ in particular, for each $n>0$, $X\in\text{Prod}(Q)^{[-m,0]}$, and family $\{C_j\}_{j\in J}$ of objects in $\text{Prod}(Q)^{[-m,0]}$, there exists a monomorphism
            \[
                \big(\extFunc_{[-m,0]}\big)^n\bigg(X,\prod_{j\in J}C_j\bigg)\hookrightarrow \prod_{j\in J}\big(\extFunc_{[-m,0]}\big)^n(X,C_j),
            \]
            which implies that $\text{Prod}(X^{\perp_{>0}})=X^{\perp_{>0}}$.
    \end{itemize}
\end{Cor}

\begin{Rmk}\label{rmk: coproducers}\leavevmode
    \begin{enumerate}[label=\RmkLbl]
    
        \item One can similarly verify an analogous construction to Corollary~\ref{cor: construction} for coproducts: if $\triangulatedTriple$ is a triangulated category with coproducts, $A\in\triangulatedCat$ is such that $\text{Add}(A)$ is a presilting subcategory of $\triangulatedCat$, and $m\in\mathbb{Z}_{>0}$, then
            \[
                \big(\text{Add}(A)^{[0,m]},\extFunc_{[0,m]},\additiveReal_{[0,m]}\big)
            \] 
            has an extriangulated structure induced by restriction of $\triangulatedCat$ with
            \begin{itemize}
                \item exact coproducts (i.e., (AET3.5) in \cite[Definition~3.2]{argudín2025recollements});
                \item both enough $\extFunc_{[0,m]}$-projectives and $\extFunc_{[0,m]}$-injectives, given by $\text{Add}(A)$ and $\shiftFunc[m]{\text{Add}(A)}$, respectively, as well as an $\extFunc_{[0,m]}$-projective generator $A$;
                \item finite positive global dimension at most $m$.
            \end{itemize}
            Moreover, for each $n>0$, $X\in\text{Add}(A)^{[0,m]}$, and family $\{C_j\}_{j\in J}$ of objects in $\text{Add}(A)^{[0,m]}$, there exists a monomorphism
            \[
                \big(\extFunc_{[0,m]}\big)^n\bigg(\coprod_{j\in J}C_j, X\bigg) \hookrightarrow \prod_{j\in J}\big(\extFunc_{[0,m]}\big)^n(C_j,X),
            \] 
            which implies that $\text{Add}({}^{\perp_{>0}}X)={}^{\perp_{>0}}X$.

        \item In \cite[]{gorsky2023hereditary}, 0-Auslander extriangulated categories were introduced and studied, and work in progress was reported for $d$-\emph{Auslander extriangulated categories} with $d\in\mathbb{Z}_{\ge0}$, a generalization consisting of an extriangulated category $\extriangulatedTriple$ with 
            \begin{itemize}
            
                \item both enough $\extFunc$-projectives and $\extFunc$-injectives;

                \item $\extFunc$-projective and $\extFunc$-injective dimension $\le d+1$;

                \item dominant and codominant dimension $\ge d+1$ (see \cite[Definition~3.5]{gorsky2023hereditary}).
            \end{itemize}
            We note that, for any presilting subcategory $\subcat$ of a triangulated category $\triangulatedCat$ and $m\in\mathbb{Z}_{>0}$, the extriangulated category $(\subcat^{[-m,0]},\extFunc_{[-m,0]},\additiveReal_{[-m,0]})$ from Proposition~\ref{prp: construction} is $(m-1)$-Auslander. In fact, the first two properties above follow directly from the result, whereas the last one follows from the equalities $\text{Proj}_{\extFunc_m}(\extriangulatedCat_m)=\shiftFunc[-m]{\subcat}$ and $\text{Inj}_{\extFunc_m}(\extriangulatedCat_m)=\subcat$ since, for any $P\in\shiftFunc[-m]{\subcat}$, the trivial distinguished triangle given by its identity morphism induces $\additiveReal_{[-m,0]}$-conflations $P\to 0\to \shiftFunc[1]{P}, \shiftFunc[1]{P}\to 0\to \shiftFunc[2]{P},\dots,\shiftFunc[m-1]{P}\to 0\to \shiftFunc[m]{P}$ via rotation, and dually for $\extFunc_{[-m,0]}$-injectives. Analogously, the category $\subcat^{[0,m]}$ is also $(m-1)$-Auslander.
    \end{enumerate}
\end{Rmk}

Corollary~\ref{cor: construction} applies to any triangulated category with products and a cosilting object. 

\begin{Dfn}\label{dfn: cosilting}\cite[Definition~1.3.1]{beilinson1982fascieaux}\cite[Definition~4.1]{psaroudakis2018realisation}\cite[Definition~3.2]{hügel2025mutation}
    Let $\triangulatedTriple$ be a triangulated category with products. 
    \begin{enumerate}[label=\DfnLbl]
        \item A $t$-\emph{structure} in $\triangulatedCat$ is a pair $\torsionPair$ of subcategories of $\extriangulatedCat$ such that $\torsionClass={}^{\perp_0}\torsionfreeClass$, $\torsionClass^{\perp_0}=\torsionfreeClass$, $\shiftFunc[1]{\torsionClass}\subseteq\torsionClass$ and, for each $C\in\triangulatedCat$, there exists a distinguished triangle $T\to C\to F\to \shiftFunc[1]{T}$ with $T\in\torsionClass$ and $F\in\torsionfreeClass$. Its \emph{heart} is then $\shiftFunc[-1]{\torsionClass}\cap\torsionfreeClass$.

        \item An object $\omega\in\triangulatedCat$ is \emph{cosilting} if $({}^{\perp_{\le0}}\omega,{}^{\perp_{>0}}\omega)$ is a $t$-structure in $\triangulatedCat$. We say that two cosilting objects $\omega$ and $\omega'$ in $\triangulatedCat$ are \emph{equivalent} if $\text{Prod}(\omega)=\text{Prod}(\omega')$.

        \item Let $\omega,\omega'\in\triangulatedCat$ be cosilting objects and $\mathcal{D}=\text{Prod}(\omega)\cap\text{Prod}(\omega')$. We say that
            \begin{itemize}
                \item[(i)] $\omega'$ is a \emph{left mutation} of $\omega$ if there exists a distinguished triangle
                    \[
                        \omega\xrightarrow[]{f}B_0\to B_1\to \shiftFunc[1]{\omega}
                    \] 
                    such that $f$ is a left $\mathcal{D}$-approximation of $\omega$ and $\text{Prod}(B_0\bigoplus B_1)=\text{Prod}(\omega')$.

                \item[(ii)] $\omega'$ is a \emph{right mutation} of $\omega$ if there exists a distinguished triangle
                    \[
                        \shiftFunc[-1]{\omega}\to A_1\to A_0\xrightarrow[]{g}\omega
                    \] 
                    such that $g$ is a right $\mathcal{D}$-approximation of $\omega$ and $\text{Prod}(A_1\bigoplus A_0)=\text{Prod}(\omega')$.
            \end{itemize}
    \end{enumerate}
\end{Dfn}

\begin{Exp}\label{exp: cosilting}
    Let $\triangulatedTriple$ be a triangulated category with products and a cosilting object $\omega$, and let $m\in\mathbb{Z}_{>0}$. Clearly, $\text{Prod}(\omega)$ is a presilting subcategory of $\triangulatedCat$. By Corollary~\ref{cor: construction}, $(\text{Prod}(\omega)^{[-m,0]},\extFunc_{[-m,0]},\additiveReal_{[-m,0]})$ is then an extriangulated category with exact products, an $\extFunc_{[-m,0]}$-injective cogenerator $\omega$, finite positive global dimension at most $m$\textemdash hence, enough $\extFunc_{[-m,0]}$-injectives and all objects with finite $\extFunc_{[-m,0]}$-injective dimension\textemdash and moreover enough $\extFunc_{[-m,0]}$-projectives. Note that $\text{Prod}(\omega)$ is a silting subcategory of $\text{Prod}(\omega)^{[-m,0]}$.
\end{Exp}

\begin{Rmk}\label{rmk: AIR}
    Similar results to Lemma~\ref{lmm: construction}(b) and Proposition~\ref{prp: construction}(c) were obtained independently in \cite[]{wei2026air}, where the authors focused on the so called $d$-\emph{extended heart} $\shiftFunc[d-1]{\mathcal{H}}\ast\dots\ast\shiftFunc[1]{\mathcal{H}}\ast\mathcal{H}$ of a $t$-structure generated by a particular type of contravariantly finite presilting subcategory of an arbitrary triangulated category.
\end{Rmk}

\subsection{$n$-Cosilting complexes}\label{ssec: n-cosilting}
This section focuses on the basic fact that $Q$ is a cosilting object in $\derivedCat(\AUR)$, which will lead us to the mutation of $n$-cosilting complexes over $R$ via large silting mutation.

\begin{Rmk}\label{rmk: Mod(R)}
    By \cite[Example~2.13]{nakaoka2019extriangulated}, $\text{Mod}(\AUR)$ is an extriangulated category with additive bifunctor $\text{Ext}_{\AUR}^{1}$ whose conflations are given by short exact sequences; in particular, deflations are given by epimorphisms, with their kernels being the respective cocones. Moreover, it has both enough ($\text{Ext}_{\AUR}^{1}$-)projectives and ($\text{Ext}_{\AUR}^{1}$-)injectives \textemdash \ hence, positive extensions, given by $\text{Ext}_{\AUR}^{n}(\ast,-)=\text{Ext}_\AUR^1(\ast,\cosyz[k-1]{-})\simeq\text{Ext}_\AUR^1(\syz[k-1]{\ast},-)$ for each $n>1$ \textemdash, as well as exact products and an ($\text{Ext}_{\AUR}^{1}$-)injective cogenerator $Q$. In the following, we consider $\text{Mod}(\AUR)$ with this extriangulated structure.
\end{Rmk}

Recall that $\text{Inj}(\AUR)$ is a presilting subcategory of $\text{Mod}(\AUR)$. By identifying $\text{Mod}(\AUR)$ with the complexes in $\derivedCat(\AUR)$ concentrated in degree zero, it follows that $\text{Inj}(\AUR)$ is a presilting subcategory of $\derivedCat(\AUR)$. Moreover, since $\text{Prod}(Q)=\text{Inj}(\AUR)$, we have by Corollary~\ref{cor: construction} that $(\text{Inj}(\AUR)^{[-n,0]},\extFunc_{[-n,0]},\additiveReal_{[-n,0]})$ is an extriangulated category which satisfies all of the hypotheses considered throughout Section~\ref{sec: large}. Note that $\text{Inj}(\AUR)^{[-n,0]}$ is equivalent to $\Kzn{\AUR}$; we denote the corresponding extriangulated structure by $(\Kzn{\AUR},\extFunc,\additiveReal)$. In particular, by Proposition~\ref{prp: bijections}, we know that product-closed silting subcategories of $\Kzn{\AUR}$ are in bijection with equivalence classes of their producers. We shall now see that, in this context, such producers coincide precisely with $n$-cosilting complexes over $\AUR$.

\begin{Prp-Dfn}\label{prp-dfn: cosilting}\cite[Definition~2.4,~Theorem~2.5~\&~Proposition~2.8]{zhang2017cosilting}
    A complex $\omega\in\derivedCat(\AUR)$ is \emph{cosilting} if it satisfies the following conditions.
    \begin{enumerate}
        \item[(i)] $\text{Hom}_{\derivedCat(\AUR)}(\text{Prod}(\omega),\shiftFunc[i]{\text{Prod}(\omega)})=0$ for all $i>0$.
        \item[(ii)] The smallest triangulated subcategory of $\derivedCat(\AUR)$ containing $\text{Prod}(\omega)$ is $\KbI{\AUR}$.
    \end{enumerate}
    We may also replace (ii) with
    \begin{enumerate}
        \item[(ii')] $\omega\in \text{Prod}(Q)^\lor$ and $Q\in\text{Prod}(\omega)^\land$.
    \end{enumerate}
    If, furthermore, $\omega\in\text{Prod}(Q)^\lor_n$ (equivalently, $Q\in\text{Prod}(\omega)^\land_n$), it is an $n$-\emph{cosilting complex}. We say that two cosilting complexes $\omega$ and $\omega'$ are \emph{equivalent} if $\text{Prod}(\omega)=\text{Prod}(\omega')$.
\end{Prp-Dfn}

\begin{Rmk}\label{rmk: cosilting}\leavevmode
    \begin{enumerate}[label=\RmkLbl]
        \item Up to shifts, cosilting complexes may be considered to be concentrated in degrees greater than or equal to zero. We adopt this convention in the following. 
        \item Condition (ii) in Proposition-Definition~\ref{prp-dfn: cosilting} implies that $\omega\in\KbI{\AUR}$, which is also equivalent to the first inclusion in condition (ii').
        \item Analogously, following (1), the inclusion $\omega\in\text{Prod}(Q)^\lor_n$ is equivalent to $\omega\in\Kzn{\AUR}$. For this reason, $n$-cosilting complexes are also referred to as $(n+1)$-\emph{term cosilting complexes} in the literature.

        \item Since, for all $A,C\in\KbI{\AUR}$, we have that
            \[
            \text{Hom}_{\KbI{\AUR}}(C,A)\simeq \text{Hom}_{\derivedCat(\AUR)}(C,A),
            \] 
            we may regard $\KbI{\AUR}$ as a subcategory of $\derivedCat(\AUR)$.
        \item By the dual of \cite[Proposition~4.2]{hügel2015silting}, all cosilting complexes are cosilting objects in $\derivedCat(\AUR)$.
    \end{enumerate}
\end{Rmk}

The viewpoint we wish to adopt is that a(n $n$-)cosilting complex is an object whose closure under products and summands is a silting subcategory of an appropriate extriangulated category. Before proceeding to the $n$-cosilting case, we show the general case. 

\begin{Prp}\label{prp: cosilting complex}
    Let $\omega\in\derivedCat(\AUR)$. Then $\omega$ is a cosilting complex if, and only if, $\text{Prod}(\omega)$ is a silting subcategory of $\KbI{\AUR}$.
\end{Prp}
\begin{proof}
    Let $\text{tria}(\text{Prod}(\omega))$ denote the smallest triangulated subcategory of $\derivedCat(\AUR)$ which contains $\text{Prod}(\omega)$. Since $\derivedCat(\AUR)$ has products, by the Eilenberg-Mazur swindle it follows that $\text{tria}(\text{Prod}(\omega))$ is closed under summands in $\derivedCat(\AUR)$ (c.f. \cite[Lemma~3.6]{hügel2019silting} for the dual case). By Remark~\ref{rmk: subcategories}(2), together with the minimality properties of both $\text{tria}(\text{Prod}(\omega))$ and $\text{thick}(\text{Prod}(\omega))$ with respect to subcategories of $\derivedCat(\AUR)$, we have that $\text{tria}(\text{Prod}(\omega))=\text{thick}(\text{Prod}(\omega))$. The assertion then follows from Remark~\ref{rmk: cosilting}(2)~\&~(4).
\end{proof}

\begin{Lmm}\label{lmm: n-cosilting}
    Let $\omega\in\Kzn{\AUR}$. Then $\text{Prod}(\omega)$ is a presilting subcategory of $\KbI{\AUR}$ if, and only if, it is a presilting subcategory of $\Kzn{\AUR}$.
\end{Lmm}
\begin{proof}
    Since $\Kzn{\AUR}$ is closed under products and summands in $\KbI{\AUR}$, $\text{Prod}(\omega)$ contains the same objects whether we consider it as a subcategory of $\KbI{\AUR}$ or of $\Kzn{\AUR}$. The assertion then follows from Remark~\ref{rmk: cosilting}(4) and Proposition~\ref{prp: construction}(c).
\end{proof}

\begin{Prp}\label{prp: n-cosilting}
    Let $\omega\in\derivedCat(\AUR)$. Then $\omega$ is an $n$-cosilting complex if, and only if, $\text{Prod}(\omega)$ is a silting subcategory of $\Kzn{\AUR}$.
\end{Prp}
\begin{proof}
    Throughout this proof, for $\subcat\subseteq \Kzn{\AUR}\subseteq \KbI{\AUR}$, let $\text{thick}(\subcat)$ denote the corresponding thick subcategory of $\Kzn{\AUR}$.

    Let $\omega$ be an $n$-cosilting complex. Then $\omega\in\Kzn{\AUR}$ by Remark~\ref{rmk: cosilting}(3) and $Q\in\text{Prod}(\omega)^\land_n$ by definition. In particular, $Q\in\text{thick}(\text{Prod}(\omega))$. Since $\Kzn{\AUR}$ has exact products, we have that $\text{Prod}(Q)\subseteq\text{thick}(\text{Prod}(\omega))$. Moreover, since $\text{Prod}(Q)$ is a silting subcategory of $\Kzn{\AUR}$, it follows that $\text{thick}(\text{Prod}(\omega))=\Kzn{\AUR}$. Also, by Proposition~\ref{prp: cosilting complex} we know that $\text{Prod}(\omega)$ is a presilting subcategory of $\KbI{\AUR}$. Thus, it follows from Lemma~\ref{lmm: n-cosilting} that $\text{Prod}(\omega)$ is a silting subcategory of $\Kzn{\AUR}$.

    Suppose that $\text{Prod}(\omega)$ is a silting subcategory of $\Kzn{\AUR}$. In particular, this implies that $\text{Prod}(\omega)$ is a presilting subcategory of $\KbI{\AUR}$. Since every complex in $\KbI{\AUR}$ is an iterated extension of shifts of injective $\AUR$-modules, by Proposition~\ref{prp: cosilting complex}, it suffices to show that any thick subcategory of $\KbI{\AUR}$ which contains $\text{Prod}(\omega)$ also contains $\text{Prod}(Q)$. However, this follows from our assumption that $\text{Prod}(\omega)$ is silting in $\Kzn{\AUR}$ and the fact that extensions, cones, cocones and summands in $\Kzn{\AUR}$ coincide with extensions, cones, cocones and summands taken in $\KbI{\AUR}$.
\end{proof}

\begin{Cor}\label{cor: bijections}
    There exists a commutative diagram of bijections
    \begin{center}
        \adjustbox{scale=1.07}
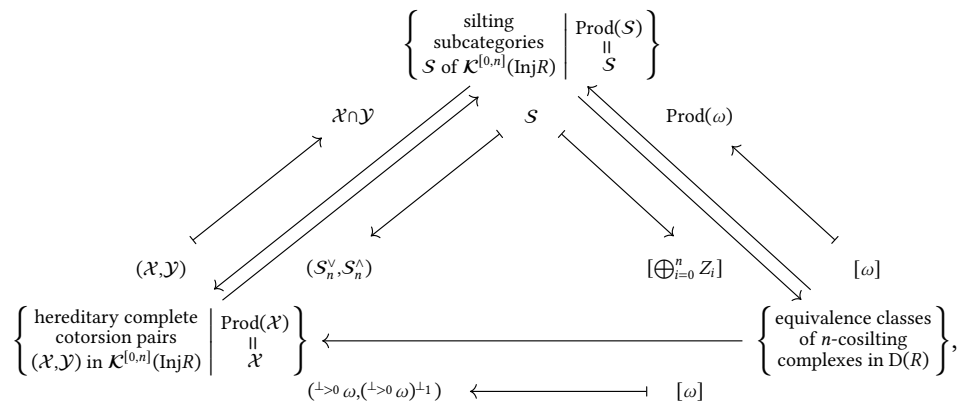
{\begin{tikzcd}
            &\bigg\{ \substack{ \text{silting} \\ \text{subcategories} \\ \subcat \text{ of } \Kzn{\AUR} } \ \bigg| \ \substack{ \text{Prod}(\subcat) \\ \rotatebox{90}{=} \\ \subcat } \bigg\} \arrow[shift right]{ddddl}[]{} \arrow[shift right]{ddddr}[]{} \arrow[phantom, very near start, xshift=11mm]{ddddl}[]{\scriptstyle \subcat} \arrow[phantom, very near end, xshift=11mm]{ddddl}[]{\scriptstyle (\subcat^\lor_n, \subcat^\land_n)} \arrow[mapsto, xshift=-9.5mm, yshift=1mm, shorten=9.25mm]{ddddr}[]{} \arrow[mapsto, xshift=11mm, yshift=1mm, shorten=10mm]{ddddl}[]{}  \\ \\ \\ \\
            \bigg\{\substack{ \text{hereditary complete} \\ \text{cotorsion pairs} \\  \cotorsionPair \text{ in } \Kzn{\AUR} } \ \bigg| \ \substack{  \text{Prod}(\cotorsionClass) \\ \rotatebox{90}{=} \\ \cotorsionClass } \bigg\} \arrow[shift right]{uuuur}[]{} \arrow[phantom, very near start, xshift=-11mm]{uuuur}[]{\scriptstyle \cotorsionPair} \arrow[phantom, very near end, xshift=-11mm]{uuuur}[]{\scriptstyle \cotorsionKernel} \arrow[mapsto, very near end, xshift=-11mm, yshift=0.5mm, shorten=10mm]{uuuur}[]{} 
            &&\bigg\{\substack{ \text{equivalence classes} \\ \text{of } n\text{-cosilting} \\ \text{complexes in } \derivedCat(\AUR) }\bigg\} \arrow[]{ll}[]{} \arrow[shift right]{uuuul}[]{} \arrow[phantom, very near start, xshift=-11.5mm]{uuuul}[]{\scriptstyle [\bigoplus_{i=0}^{n}Z_i]} \arrow[phantom, very near start, shift left=10mm, yshift=3.75mm]{ll}[]{\scriptstyle [\omega]} \arrow[phantom, very near end, shift left=10mm, yshift=3.75mm]{ll}[]{\scriptstyle ({}^{\perp_{>0}}\omega,({}^{\perp_{>0}}\omega)^{\perp_1})} \arrow[mapsto, very near start, shift left=10mm, shorten=15mm, xshift=3.25mm, yshift=3.75mm]{ll}[]{} \arrow[phantom, very near start, xshift=11mm]{uuuul}[]{\scriptstyle [\omega]} \arrow[phantom, very near end, xshift=11mm]{uuuul}[]{\scriptstyle \text{Prod}(\omega)} \arrow[mapsto, xshift=11mm, shorten=10mm]{uuuul}[]{}, \\
        \end{tikzcd}}
    \end{center}
    where the $Z_i$ are calculated as follows: given a silting subcategory $\subcat$ as in the upper collection, consider its corresponding cotorsion pair $(\subcat^\lor_n,\subcat^\land_n)$ and an $\subcat$-resolution of $Q$ of length $n$ 
    \begin{align*}
        L_1\to &Z_0\to Q, \\
        L_2\to &Z_1\to L_1, \\
        L_3\to &Z_2\to L_2, \\
               &\vdots \\
        0\to &Z_n\to L_n.
    \end{align*}
\end{Cor}
\begin{proof} 
    Since $\Kzn{\AUR}$ has finite positive global dimension $n$, it follows from Lemma~\ref{lmm: dimension} that $(\subcat^\lor,\subcat^\land)=(\subcat^\lor_n,\subcat^\land_n)$ for each silting subcategory $\subcat$ of $\Kzn{\AUR}$. Moreover, since it has both enough $\extFunc$-projectives and $\extFunc$-injectives, all cotorsion pairs in $\Kzn{\AUR}$ are bounded by Proposition~\ref{prp: bounded}. Thus, Proposition~\ref{prp: bijections} and Proposition~\ref{prp: n-cosilting} give us mutually inverse bijections
    \begin{center}
        \adjustbox{scale=1.04}{\begin{tikzcd}
            &\bigg\{ \substack{ \text{silting} \\ \text{subcategories} \\ \subcat \text{ of } \Kzn{\AUR} } \ \bigg| \ \substack{ \text{Prod}(\subcat) \\ \rotatebox{90}{=} \\ \subcat } \bigg\} \arrow[shift right]{ddddl}[]{} \arrow[shift right]{ddddr}[]{} \arrow[phantom, very near start, xshift=11mm]{ddddl}[]{\scriptstyle \subcat} \arrow[phantom, very near end, xshift=11mm]{ddddl}[]{\scriptstyle (\subcat^\lor_n, \subcat^\land_n)} \arrow[mapsto, xshift=-9.5mm, yshift=1mm, shorten=9.25mm]{ddddr}[]{} \arrow[mapsto, xshift=11mm, yshift=1mm, shorten=10mm]{ddddl}[]{}  \\ \\ \\ \\
            \bigg\{\substack{ \text{hereditary complete} \\ \text{cotorsion pairs} \\  \cotorsionPair \text{ in } \Kzn{\AUR} } \ \bigg| \ \substack{  \text{Prod}(\cotorsionKernel) \\ \rotatebox{90}{=} \\ \cotorsionKernel } \bigg\} \arrow[shift right]{uuuur}[]{} \arrow[phantom, very near start, xshift=-11mm]{uuuur}[]{\scriptstyle \cotorsionPair} \arrow[phantom, very near end, xshift=-11mm]{uuuur}[]{\scriptstyle \cotorsionKernel} \arrow[mapsto, very near end, xshift=-11mm, yshift=0.5mm, shorten=10mm]{uuuur}[]{} 
            &&\bigg\{\substack{ \text{equivalence classes} \\ \text{of } n\text{-cosilting} \\ \text{complexes in } \derivedCat(\AUR) }\bigg\} \arrow[shift right]{uuuul}[]{} \arrow[phantom, very near start, xshift=-11.5mm]{uuuul}[]{\scriptstyle [\bigoplus_{i=0}^{n}Z_i]} \arrow[phantom, very near start, xshift=11mm]{uuuul}[]{\scriptstyle [\omega]} \arrow[phantom, very near end, xshift=11mm]{uuuul}[]{\scriptstyle \text{Prod}(\omega)} \arrow[mapsto, xshift=11mm, shorten=10mm]{uuuul}[]{},
        \end{tikzcd}}
    \end{center}
    with the $Z_i$ calculated as in the above statement. By Lemma~\ref{lmm: equivalence}, the property in the bottom left corner is equivalent to $\text{Prod}(\cotorsionClass)=\cotorsionClass$. Finally, since $\Kzn{\AUR}$ has an $\extFunc$-injective cogenerator, the bottom map in the statement follows from Proposition~\ref{prp: generated}.
\end{proof}

\begin{Rmk}\label{rmk: dga}
    An analogous result may be obtained by replacing $\AUR$ and $\text{Hom}_\mathbb{Z}(\AUR,\mathbb{Q}/\mathbb{Z})$ with a connective differential graded (dg) algebra $A$ (i.e., a dg algebra such that $A^i=0$ for all $i>0$) and its character dual $A^+$, since the latter is then a cosilting object in the derived category of $A$, by which we can apply Corollary~\ref{cor: construction} using $\text{Prod}(A^+)$. In this case, the objects in $\text{Prod}(A^+)$ are sometimes called \emph{derived injectives} (see \cite[]{genovese2021t-structures} for reference).
\end{Rmk}

\subsection{Large $n$-(co)tilting modules over rings of finite global dimension}\label{ssec: n-(co)tilting}

We prove that, if $\AUR$ has finite global dimension $n$, then Theorem~\ref{thm: A} and Theorem~\ref{thm: B} can be applied directly to $\text{Mod}(\AUR)$, in which case producers coincide with large $n$-cotilting modules.

\begin{Prp-Dfn}\label{prp-dfn: n-cotilting}\cite[\S 2]{hügel2001infinitely}\cite[Definition~4~\&~Proposition~3.5]{bazzoni2004characterization}
    A module $C\in\text{Mod}(\AUR)$ is \emph{cotilting} if it satisfies the following conditions.
    \begin{itemize}
        \item[(i)] $\text{id}(C)<\infty$.
        \item[(ii)] $\text{Ext}_{\AUR}^{i}(C^J,C)=0$ for all $i>0$ and sets $J$.
        \item[(iii)] There exists an injective cogenerator $Q$ of $\text{Mod}(\AUR)$ such that $Q\in\text{Prod}(C)^\land$.
    \end{itemize}
    If, furthermore, $\text{id}(C)\le n$ (which implies that $Q\in\text{Prod}(C)^\land_n$), it is an $n$-\emph{cotilting module}. We say that two cotilting modules $C$ and $C'$ are \emph{equivalent} if $\text{Prod}(C)=\text{Prod}(C')$.
\end{Prp-Dfn}

\begin{Prp}\cite[Theorem~4.2]{hügel2001infinitely}\label{prp: n-cotilting}
    Let $\cotorsionClass\subseteq\text{Mod}(\AUR)$ be closed under cocones in $\text{Mod}(\AUR)$ and such that $\cotorsionClass\cap\cotorsionClass^{\perp_{>0}}$ is closed under products in $\text{Mod}(\AUR)$. The following conditions are equivalent.
    \begin{enumerate}[label=\ThmLbl]
        \item There exists an $n$-cotilting module $C\in\text{Mod}(\AUR)$ such that $\cotorsionClass={}^{\perp_{>0}}C$.
        \item $\text{id}(\cotorsionClass^{\perp_{>0}})\le n$ and, for all $M\in\text{Mod}(\AUR)$, there exists an epimorphic right $\cotorsionClass$-approximation of $M$ with kernel in $\cotorsionClass^{\perp_1}$.
    \end{enumerate}
\end{Prp}

The following is a rephrasing of \cite[Theorem~2.4]{saroch2007completeness}.

\begin{Lmm}[]\label{lmm: n-cotilting}
    Let $C\in\text{Mod}(\AUR)$ be an $n$-cotilting module. Then the cotorsion pair $\cotorsionPair$ in $\text{Mod}(\AUR)$ generated by ${}^{\perp_{>0}}C$ is complete, hereditary, and such that $\text{Prod}(\cotorsionKernel)=\cotorsionKernel$. In particular, $\cotorsionPair=({}^{\perp_{>0}}C,({}^{\perp_{>0}}C)^{\perp_1})$ and $\cotorsionKernel=\text{Prod}(C)$.
\end{Lmm}
\begin{proof}
    Let $C$ be an $n$-cotilting module. Note that ${}^{\perp_{>0}}C$ is closed under cocones in $\text{Mod}(\AUR)$ by Remark~\ref{rmk: heredity}(1), whereas \cite[Lemma~2.3(b)~\&~Lemma~2.4(b)(ii)]{hügel2001infinitely} imply that ${}^{\perp_{>0}}C\cap({}^{\perp_{>0}}C)^{\perp_{>0}}=\text{Prod}(C)$. Thus, we can apply Proposition~\ref{prp: n-cotilting} to obtain that, for all $M\in\text{Mod}(\AUR)$, there exists an epimorphic right ${}^{\perp_{>0}}C$-approximation of $M$ with kernel in $({}^{\perp_{>0}}C)^{\perp_1}$. Now, by Lemma~\ref{lmm: generated'}, the cotorsion pair in $\text{Mod}(\AUR)$ generated by ${}^{\perp_{>0}}C$ is hereditary and equal to $({}^{\perp_{>0}}C,({}^{\perp_{>0}}C)^{\perp_1})$. It then follows from Salce's lemma that it is also complete. Moreover, Lemma~\ref{lmm: generated} implies that ${}^{\perp_{>0}}C\cap({}^{\perp_{>0}}C)^{\perp_{1}}=\text{Prod}(C)$, which completes the proof.
\end{proof}

\begin{Prp}[]\label{prp: n-cotilting'}
    Let $C\in\text{Mod}(\AUR)$. Then $C$ is an $n$-cotilting module if, and only if, $\text{Prod}(C)$ is an $n$-cotilting subcategory of $\text{Mod}(\AUR)$.
\end{Prp}
\begin{proof}
    Suppose that $\text{Prod}(C)$ is an $n$-cotilting subcategory of $\text{Mod}(\AUR)$. Note that: $\text{Prod}(C)$ being a presilting subcategory of $\text{Mod}(\AUR)$ implies that $\text{Ext}_{\AUR}^{i}(C^J,C)=0$ for all $i>0$ and sets $J$; the inclusion $\text{Prod}(C)\subseteq\text{Inj}(\AUR)^\lor_n$ in particlar implies that $\text{id}(C)\le n$; the inclusion $\text{Inj}(\AUR)\subseteq\text{Prod}(C)^\land_n$ in particular implies that the injective cogenerator $\text{Hom}_\mathbb{Z}(\AUR,\mathbb{Q}/\mathbb{Z})\in\text{Prod}(C)^\land$. Thus, $C$ is an $n$-cotilting module.

    Conversely, suppose that $C$ is an $n$-cotilting module. It follows from Proposition-Definition~\ref{prp-dfn: n-cotilting}(ii) that $\text{Prod}(C)$ is a presilting subcategory of $\text{Mod}(\AUR)$. Now, consider the hereditary cotorsion pair $({}^{\perp_{>0}}C,({}^{\perp_{>0}}C)^{\perp_1})$ from Lemma~\ref{lmm: n-cotilting}. Then, by \cite[Lemma~8.1.4 (a)~\&~(c)]{göbel2006approximations}, we have that $\text{Prod}(C)\subseteq\text{Inj}(\AUR)^\lor_n$, whereas \cite[Proposition~8.1.5]{göbel2006approximations} implies that $({}^{\perp_{>0}}C)^{\perp_1}=\text{Prod}(C)^\land_n$, from which it follows that $\text{Inj}(\AUR)\subseteq\text{Prod}(C)^\land_n$. 
\end{proof}

\begin{Cor}[]\label{cor: n-cotilting}
    If $\AUR$ has finite global dimension $n$, then $C\in\text{Mod}(\AUR)$ is an $n$-cotilting module if, and only if, $\text{Prod}(C)$ is a silting subcategory of $\text{Mod}(\AUR)$.
\end{Cor}
\begin{proof}
    This follows directly from Proposition~\ref{prp: n-cotilting'} and the dual of Proposition~\ref{prp: n-(co)tilting}.
\end{proof}

\begin{Rmk}\label{rmk: analog}
    The previous result tells us that, in the context of $\text{Mod}(\AUR)$ for a given ring $\AUR$ of finite global dimension $n$, the producers of the product-closed silting subcategories of $\text{Mod}(\AUR)$ coincide precisely with the $n$-cotilting modules over $\AUR$. Note that in this case all objects in $\text{Mod}(\AUR)$ have finite injective dimension, and moreover $\text{Mod}(\AUR)$ has positive extensions, exact products, an injective cogenerator and enough projectives by Remark~\ref{rmk: Mod(R)}. It follows that we can easily obtain an analog of Corollary~\ref{cor: bijections}, with the bottom right corner being the collection of equivalence classes of $n$-cotilting modules, essentially by swapping Proposition~\ref{prp: n-cosilting} for Corollary~\ref{cor: n-cotilting} in the proof. Moreover, by Theorem~\ref{thm: A}, in this case we can mutate $n$-cotilting modules directly in $\text{Mod}(\AUR)$ via large silting mutation.

    By similar arguments, if $\AUR$ has finite global dimension $n$, then analogs of Theorem~\ref{thm: A} and Theorem~\ref{thm: B} for coproduct-closed silting subcategories can be applied to $\text{Mod}(\AUR)$, in which case their coproducers coincide with large $n$-tilting modules.
\end{Rmk}

We end by exhibiting a phenomenon of large silting mutation which does not occur in the usual setting, and which appears in the well-known case of the Kronecker algebra.

\begin{Exp}\label{exp: n-cotilting}
    Let
    \begin{itemize}
        \item $\ACF$ be an algebraically closed field;
        \item $A$ denote the path $\ACF$-algebra of the Kronecker quiver \begin{tikzcd} \bullet \arrow[shift left]{r}[]{}\arrow[shift right]{r}[]{} &\bullet\end{tikzcd};
        \item $\mathbb{X}$ denote the projective line over $\ACF$;
        \item $P$ be a subset of $\mathbb{X}$.
    \end{itemize}
    Recall that the finite-dimensional indecomposable regular ($A$-)modules can be partitioned into a family $\{\mathbf{t}_x\}_{x\in\mathbb{X}}$ of tubes indexed over $\mathbb{X}$, with the adic and Prüfer modules corresponding to $\mathbf{t}_x$ respectively denoted by $S_x[-\infty]$ and $S_x[\infty]$, and that the generic module $G$ satisfies $\text{Prod}(G)=\text{Add}(G)$ due to being endofinite. It follows from \cite[Corollary~3.10]{buan2003cotilting} and \cite[Example~4.10]{hügel2025mutation} that all large (1-)cotilting modules are of the form
    \[
        C_P = \bigg(\prod_{x\in P}S_x[-\infty]\bigg)\bigoplus G\bigoplus \bigg(\prod_{y\notin P}S_{y}[\infty]\bigg)
    \] 
    up to equivalence, where $\text{Prod}(G)\subseteq\text{Prod}(S_x[\infty])$ for any $x\in\mathbb{X}$ by \cite[Proposition~3]{ringel1998ziegler}. 

    Now observe that, if some $C_P$ was mutable and its mutation $C_P'$ contained any Prüfer module as a summand, then $C_P'$ would also contain the generic module as a summand (up to equivalence). Thus, the summand $G$ would remain fixed (up to equivalence) during such a mutation, indicating that it might have the approximation properties which enable the mutation itself. This is precisely what occurs in the mutations of this example: by \cite[Lemma~2.4]{buan2003cotilting}, for each $x\in\mathbb{X}$, there exist a set $J_x$ and a short exact sequence
    \[
        0\to S_x[-\infty]\xrightarrow[]{f_x} G^{J_x}\xrightarrow[]{g_x} S_x[\infty]\to 0
    \] 
    with $f_x$ a left $\text{Prod}(G)$-approximation of $S_x[-\infty]$ and $g_x$ a right $\text{Prod}(G)$-approximation of $S_x[\infty]$; since $\text{Mod}(\AUR)$ has exact products, we can combine these with trivial short exact sequences of the form $0\to S_x[-\infty]\to S_x[-\infty]\to 0\to 0$ and $0\to 0\to S_x[\infty]\to S_x[\infty]\to 0$ in order to construct exact sequences 
    \begin{align*}
        0\to S_x[-\infty] \to &G^J\bigoplus C_P/S_x[\infty]\to C_P\to 0, \tag{for $x\notin P\subsetneq\mathbb{X}$} \\
        0\to C_P \to &C_P/S_x[-\infty]\bigoplus G^{J'}\to S_x[\infty]\to 0, \tag{for $x\in P\supsetneq\varnothing$}
    \end{align*}
    with the desired approximation properties, which respectively indicate a right or left mutation of $C_P$.
\end{Exp}

\section*{Acknowledgments}\label{sec: acknowledgments}
\thanks{The author is grateful to Hiroyuki Nakaoka, Octavio Mendoza Hern\'andez and Alejandro Argud\'in Monroy, and Jan \v{S}t'ov\'i\v{c}ek for nurturing discussions about positive extensions in extriangulated categories, (co)products in extriangulated categories, and cosilting theory and dg categories, respectively.}

\printbibliography

\end{document}